\documentclass[12pt]{amsart}
\usepackage{amsmath,amssymb}
\usepackage{graphicx,bm}
\usepackage[margin=1.4in]{geometry}
\usepackage{hyperref}
\usepackage{amsthm}
\usepackage{verbatim}
\usepackage{xcolor}
\usepackage{enumitem}
\numberwithin{equation}{section}
\usepackage{slashed}
\allowdisplaybreaks

\newtheorem{thm}{Theorem}[section]

\newtheorem{cor}[thm]{Corollary}
\newtheorem{lem}[thm]{Lemma}

\newtheorem{prop}[thm]{Proposition}

\newtheorem{rmk}[thm]{Remark}

\numberwithin{equation}{section}

\newcommand*{\Rmn}[1]{\uppercase\expandafter{\romannumeral#1}}

\newcommand{\eps}{{\varepsilon}}
\newcommand{\R}{{\mathbf R}}

\newcommand{\be}{\begin{equation*}}
\newcommand{\ee}{\end{equation*}}
\newcommand{\beq}{\begin{equation}}
\newcommand{\eeq}{\end{equation}}
\newcommand{\begincal}{\begin{eqnarray*}}
\newcommand{\fincal}{\end{eqnarray*}}

\makeatletter
\@namedef{subjclassname@2020}{%
  \textup{2020} Mathematics Subject Classification}
\makeatother

\begin{document}

\title[Uniqueness of conformal-harmonic maps]{Uniqueness of conformal-harmonic maps on locally conformally flat 4-manifolds}

\author{Longzhi Lin \and Jingyong Zhu}
\address{Longzhi Lin: Mathematics Department, University of California Santa Cruz, 1156 High Street, Santa Cruz, CA 95064, USA}
\email{lzlin@ucsc.edu}

\address{Jingyong Zhu: Department of Mathematics, Sichuan University, No.24 South Section 1, Yihuan Road, Chengdu, 610065, China}
\email{jzhu@scu.edu.cn}


\begin{abstract}
Motivated by the theory of harmonic maps on Riemannian surfaces, conformal-harmonic maps between two Riemannian manifolds $\mathcal{M}$ and $\mathcal{N}$ were introduced in search of a natural notion of ``harmonicity'' for maps defined on a general even dimensional Riemannian manifold $\mathcal{M}$. They are critical points of a conformally invariant energy functional and reassemble the GJMS operators when the target is the set of real or complex numbers. On a four dimensional manifold, conformal-harmonic maps are the conformally invariant counterparts of the intrinsic bi-harmonic maps and a mapping version of the conformally invariant Paneitz operator for functions.

In this paper, we consider conformal-harmonic maps from locally conformally flat 4-manifolds into spheres. We prove a quantitative uniqueness result for such conformal-harmonic maps as an immediate consequence of convexity for the conformally-invariant energy functional. To this end, we are led to prove a version of second order Hardy inequality on manifolds, which may be of independent interest.
\end{abstract}
\maketitle

\section{Introduction}
The most prominent and classic problem in calculus of variations for mappings between two Riemannian manifolds $(\mathcal{M}, g)$ and $(\mathcal{N},h)\hookrightarrow \mathbf{R}^K$ is the study of {\textit{harmonic maps}}, which are critical points $u: \mathcal{M}\to \mathcal{N}$ of the Dirichlet energy
$$
E_1(u, \mathcal{M}):=\int_\mathcal{M} |du|^2\omega_g\,,
$$
where $\omega_g$ is the volume measure on $\mathcal{M}$ defined by the metric $g$ and $|du|^2$ is the Hilbert-Schmidt norm square of $du$. The conformal invariance of $E_1(u, \mathcal{M})$ on a Riemannian surface $\mathcal{M}$ (with respect to the metric $g$ of $\mathcal{M}$) and its connection to the theory of minimal surfaces make harmonic maps on two dimensional domains the most widely studied topic in the field of geometric analysis ever since the pioneering work of J. Eells and J. Sampson \cite{ES64}, see also \cite{He02}. Motivated by the theory of harmonic maps on Riemannian surfaces, the Paneitz operator \cite{Pa83} and GJMS operators \cite{GJMS92}, in a recent work \cite{Berard08} V. B\'{e}rard has shown the existence of an intrinsically defined energy functional for smooth maps between two Riemannian manifolds $(\mathcal{M}^{2m}, g)$ (of even dimension $2m$ where $m\geq 1$) and $\mathcal{N}$, denoted by $E_m(u, \mathcal{M})$, which is {\textit{conformally invariant}} with respect to $g$ and coincides with the above Dirichlet energy $E_1(u, \mathcal{M})$ when $m=1$. Following the terminology of A. Gastel and A. Nerf in \cite{GN13} who considered an \textit{extrinsic} analogue (i.e., a variant dependent of the embedding $\mathcal{N}\hookrightarrow \mathbf{R}^K$) of $E_m(u, \mathcal{M})$, we will call $E_m(u, \mathcal{M})$ an {\textit{intrinsic}} Paneitz energy when $m=2$ and {\textit{intrinsic}} Paneitz poly-energy when $m\geq 3$. The critical points of $E_m(u, \mathcal{M})$ are called {\textit {conformal-harmonic maps}} or {\textit{C-harmonic maps}}, which generalize the harmonic maps on surfaces (i.e., $m=1$) and satisfy a system of nonlinear PDEs with leading term $\Delta^{m}$. When $\mathcal{N}=\mathbf{R}$ or $\mathbf{C}$, the induced operators for the critical points reassemble the GJMS operators \cite{GJMS92}, see also Chang-Yang \cite{CY99}. In particular, when $m=2$ the intrinsic Paneitz energy reads as
\begin{equation}\label{energy}
\mathcal{E}(u, \mathcal{M}):=E_2(u,\mathcal{M}^4)=\int_\mathcal{M}|\tau(u)|^2+\frac23 {\rm Sc}^\mathcal{M}|du|^2-2{\rm Ric}^\mathcal{M}(du,du)\,,
\end{equation}
where $\tau(u)\in u^*T\mathcal{N}$ is the tension field, we denote the scalar curvature and Ricci curvature of $\mathcal{M}$ by ${\rm Sc}^\mathcal{M}$ and ${\rm Ric}^{\mathcal{M}}$, respectively. We remark that $\mathcal{E}(u,\mathcal{M})$ is conformally invariant on a four dimensional manifold $\mathcal{M}$ (see Appendix \ref{appA}) and critical points of $\mathcal{E}(u, \mathcal{M})$ are the conformal-invariant counterparts of the {\textit{intrinsic bi-harmonic maps}}, i.e., the critical points of the {\textit{intrinsic bi-energy}} 
$$\mathcal{F}(u, \mathcal{M}):=\int_\mathcal{M}|\tau(u)|^2\,.$$
Conformal-harmonic maps are also generalizations of the Paneitz operator in the context of maps, and {\textit{intrinsically}} they satisfy the following system of $4$-th order PDEs (c.f. \cite{Jiang86, Lamm05, Mos05}):
$$
\Delta^u \tau (u)=R^{\mathcal{N}}(\nabla u, \tau(u))\nabla u+\frac23\left({\rm Sc}^\mathcal{M}\tau(u)+\nabla {\rm Sc}^\mathcal{M}\nabla u\right)-2{\rm div}({\rm Ric}^\mathcal{M}( \nabla u,\cdot)),
$$
where $\Delta^u$ is the induced Laplace operator on the pullback vector bundle $u^\ast T\mathcal{N}$ over $\mathcal{M}$ or equivalently and {\textit{extrinsically}} (c.f. \cite[equation (1.2)]{LL21}, \cite[equation (9)]{LaR})
\begin{align}\label{ELP}
\Delta^2u  = &-\Delta(\nabla P^\bot\nabla u)-{\rm div}(\nabla P^\bot\Delta u)+2\nabla P^\bot\nabla(\nabla P^\bot\nabla u)\notag\\
&+2\nabla P^\bot\nabla P^\bot\Delta u-(\nabla PP^\bot-P^\bot\nabla P)\nabla\Delta u\notag\\
&+P\langle\nabla P^\bot\nabla u,{D}_u({D}_u P^\bot)\nabla u\nabla u\rangle-2{\rm div}\langle\nabla P^\bot\nabla u,\nabla P^\bot P\rangle\\
&+2\langle\nabla P^\bot\nabla u,\nabla P^\bot\nabla P\rangle+\frac23 {\rm Sc}^\mathcal{M}(\Delta u+\nabla P^\perp\nabla u) \notag\\
&+\frac23\nabla {\rm Sc}^{\mathcal{M}}\nabla u-2\nabla\text{Ric}^\mathcal{M}(\nabla u, \cdot)-2{\rm Ric}^\mathcal{M}(P(\nabla^2u), \cdot)\,,\notag
\end{align}
see Appendix \ref{appB} for a detailed derivation of this equation. Here $R^{\mathcal{N}}$ is the curvature tensor of $\mathcal{N}$, $\nabla$ and $\Delta$ are the Levi-Civita connection and Laplace–Beltrami operator on $\mathcal{M}$ respectively, $\tau(u)=(\Delta u)^T$ is the tension field,  $P(u)$ is the orthogonal projection to the tangent plane $T_u\mathcal{N}$, $D$ is the derivative with respect to the standard coordinates of $\mathbf{R}^K$ and $A(\cdot, \cdot)$ denotes the second fundamental form  of $\mathcal{N} \hookrightarrow \mathbf{R}^K$. Note that this conformal-harmonic map equation \eqref{ELP} differs from the intrinsic biharmonic map
equation by lower order terms which make \eqref{ELP} conformally invariant.

\vskip 3mm

Higher order geometric variational problems, including the study of (both extrinsic and intrinsic) biharmonic maps and polyharmonic maps, have attracted much attention in the last two decades, see e.g. \cite{DK98, CWY99, Strz03, Wang04, LR, Sche08, Stru08, GS, LW09, GN13, LY16, AY17, HJL19, LFG21} for the extrinsic case and \cite{Wang04, Ku08, Mos08, GS, LL21} for the intrinsic case. The corresponding heat flows have been studied extensively as well, as a tool to prove the existence of biharmonic maps and polyharmonic maps in a given homotopy class, see e.g. \cite{Lamm04, Lamm05, Gas06, Wang07, HHW14, HLLW16, LL21}. It should be noted that the extrinsic and intrinsic cases come in two different flavors: the intrinsic variants are considered more geometrically natural because they do not depend on the embedding of the target manifold $\mathcal{N}\hookrightarrow \mathbf{R}^K$, although they are less natural from the variational point of view due to the lack of {\textit{coercivity}} for the intrinsic energies (and thus they are considerably more difficult analytically and less studied); on the other hand, the extrinsic variants are more natural from the analytical point of view but in turn they do depend on the embedding of $\mathcal{N}\hookrightarrow \mathbf{R}^K$. Among them, the conformally invariant problems (both extrinsic and intrinsic) are considered the most geometric, see e.g., \cite{DK98, GN13}. Most of the results in the current literature concern the regularity and existence of (the weak solutions of) the critical points of the higher order geometric variational problems because they are associated to systems of higher order PDEs with {\textit{critical growth nonlinearities}}. However, the {\textit {uniqueness}} problem of these critical points have been left largely open. In \cite{LL21} P. Laurain and the first author proved a version of bienergy convexity (and thus the uniqueness) for weakly intrinsic biharmonic maps in $W^{2,2}(B_1, \mathbf{S}^n)$ with small bienergy and prescribed boundary data, where $B_1\subset \mathbf{R}^4$ is the unit 4-ball and $\mathbf{S}^n \subset \mathbf{R}^{n+1}$ is the standard unit sphere. We shall remark that such energy convexity plays an essential role in the discrete replacement min-max construction of geometric objects of interest (such as minimal spheres and free boundary minimal disks in manifolds), see e.g. Colding-Minicozzi \cite{CM08}, Lamm-Lin \cite{LL13}, Lin-Sun-Zhou \cite{LSZ20} and Laurain-Petrides \cite{LP19}.

We shall remark that the intrinsic Paneitz energy functional $\mathcal{E}(u, \mathcal{M})$ defined in \eqref{energy} on $\mathbf{S}^4$ was already used by T. Lamm in \cite{Lamm05} as a tool to prove that every weakly intrinsic biharmonic maps from $\mathbf{R}^4$ into a non-positively curved target manifold with finite bienergy has to be constant. More recently, O. Biquard and F. Madani in \cite{Biquard} used the corresponding heat flow for $\mathcal{E}(u, \mathcal{M})$ to prove an existence result for conformal-harmonic maps from a certain class of 4-manifolds into a non-positively curved manifold $\mathcal{N}$. Here the conformal-harmonic map heat flow (or, the negative $L^2$-gradient flow of $\mathcal{E}(\cdot, \mathcal{M})$) is defined as follow (when $\partial \mathcal{M} = \emptyset$):
\beq
\label{conformal-harmonic-heat-flow-1}
\left\{
\begin{aligned}
\frac{\partial u}{\partial t} + \Delta^u \tau (u) =&R^{\mathcal{N}}(\nabla u, \tau(u))\nabla u+\frac23\left({\rm Sc}^\mathcal{M}\tau(u)+\nabla {\rm Sc}^\mathcal{M}\nabla u\right)\\
&-2{\rm div}({\rm Ric}^\mathcal{M}( \nabla u,\cdot)) \quad \text{ on } \mathcal{M} \times [0, T) \\
u=&  \,u_0  \quad \text{ on } \mathcal{M}\times \{0\}\,,
\end{aligned}
\right.
\eeq
and (when $\partial \mathcal{M}\neq \emptyset$)
\beq
\label{conformal-harmonic-heat-flow-2}
\left\{
\begin{aligned}
\frac{\partial u}{\partial t} + \Delta^u \tau (u) =&R^{\mathcal{N}}(\nabla u, \tau(u))\nabla u+\frac23\left({\rm Sc}^\mathcal{M}\tau(u)+\nabla {\rm Sc}^\mathcal{M}\nabla u\right)\\
&-2{\rm div}({\rm Ric}^\mathcal{M}( \nabla u,\cdot)) \quad \text{ on } \mathcal{M} \times [0, T) \\
u=&  \,u_0  \quad \text{ on } \mathcal{M}\times \{0\}\\
u=&  \,u_0\,, \quad \partial_{\nu} u = \partial_{\nu} u_0 \quad  \text{ on } \partial \mathcal{M}\times [0, T)\,,
\end{aligned}
\right.
\eeq
where $u_0 \in C^{\infty}\cap W^{2,2}(\overline{\mathcal{M}},\mathcal{N})$. More precisely, in \cite{Biquard} O. Biquard and F. Madani proved
\begin{thm}
Let $(\mathcal{M},g)$ and $(\mathcal{N},h)$ be compact Riemannian manifolds of four dimensions and $n$ dimensions respectively. Assume that $\mathcal{N}$ has non-positive curvature, the Yamabe invariant
$$
\mu(\mathcal{M}, [g]):= \inf_{g'\in [g]}\frac{\int_{\mathcal{M}} \text{Sc}_{g'} dv_{g'}}{\left(\int_{\mathcal{M}} dv_{g'}\right)^{1/2}}
>0$$
and the conformally invariant total $Q$-curvature
$$
\kappa(\mathcal{M}, [g]):= \frac{1}{12}\int_{\mathcal{M}} \left( \text{Sc}^2_{g}-3|\text{Ric}_g|^2 \right)dv_g
$$
satisfies $\kappa+\frac{1}{6}\mu^2>0$. Then the conformal-harmonic map heat flow \ref{conformal-harmonic-heat-flow-1} exists and is smooth for all time and converges subsequently to a smooth conformal-harmonic map $u_\infty(\mathcal{M}, \mathcal{N})$ as $t\to \infty$. Consequently, there exists a conformal-harmonic map in each homotopy class in $C^\infty(\mathcal{M}, \mathcal{N})$.
\end{thm}

\vskip 3mm

In this paper, we consider the intrinsic Paneitz energy $\mathcal{E}(u, \mathcal{M})$ for $W^{2,2}$ maps from locally conformally flat 4-manifolds with smooth boundaries into spheres and study the (quantitative) uniqueness of the critical points. The first main result of this paper is a version of energy convexity for $\mathcal{E}(u, \mathcal{M})$, more precisely, we prove

\begin{thm}\label{convexity}
Let $(\mathcal{M},g)$ be a four dimensional compact locally conformally flat Riemannian manifold with a smooth boundary $\partial \mathcal{M}$. There exist  $\epsilon_0>0, C>0$ depending only on $\mathcal{M}$ such that for any $u,v\in W^{2,2}(\mathcal{M},\mathbf{S}^n)$ with $u=v$, $\partial_\nu u=\partial_\nu v$ on $\partial \mathcal{M}$, 
\begin{equation}\label{bienergy condition}
\int_{\mathcal{M}}|\Delta u|^2dv_g\leq \epsilon_0, \quad \int_{\mathcal{M}}|\nabla v|^4dv_g\leq\epsilon_0
\end{equation}
and $u$ is a weakly conformal-harmonic map, we have
\begin{equation}\label{convexity ineq}
    \frac{1}{C}\int_{\mathcal{M}}|\Delta v-\Delta u|^2dv_g\leq \mathcal{E}(v,\mathcal{M})-\mathcal{E}(u,\mathcal{M})\,.
\end{equation}
\end{thm}

\begin{rmk}
    Our proof relies locally on the result of Laurain and the first author \cite{LL21}, which was established on the unit ball in Euclidean space. Our results could also extend to general source and target manifolds if the Laurain-Lin result is generalized to such settings.
\end{rmk}
In order to prove this theorem, we are led to show a second order Hardy inequality on smooth manifolds with smooth boundaries. Inspired by the recent work of D'Ambrosio-Dipierro \cite{DD14} (which followed the techniques introduced by Mitidieri in \cite{Mit00}), we are able to prove the following version of second order Hardy inequality on certain smooth manifolds, for details see Section \ref{Hardyineq}.

\begin{thm}\label{secondhardy-v2}
Let $(\mathcal{M},g)$ be an $n$ dimensional compact Riemannian manifold with a smooth boundary $\partial \mathcal{M}$ and $\rho \in W^{1,2}_{loc}(\mathcal{M}, \mathbf{R})$ be a nonnegative function such that $\Delta \rho \leq 0$ on $\mathcal{M}$ in the weak sense and $|\nabla \rho|\geq \delta$ a.e. on $\mathcal{M}$ for some $\delta>0$, then there exits $C:=C(\mathcal{M})>0$ such that
$$\int_{\mathcal{M}}\frac{|w|^2|\nabla \rho|^2}{\rho^4} dv_g \leq C(1+\delta^{-2})^2\int_{\mathcal{M}} |\Delta w|^2 dv_g$$
for any $w\in C^{\infty}_0(\mathcal{M}, \mathbf{R})$.
\end{thm}

In fact, based on Theorem \ref{secondhardy-v2} we are able to prove a general version of second order Hardy inequality that is valid on \textit{any} smooth manifold.

\begin{thm}\label{secondhardy}
Let $(\mathcal{M},g)$ be an $n$ dimensional compact Riemannian manifold with a smooth boundary $\partial \mathcal{M}$ and $\rho$ be the distance function to the boundary. Then there exits $C:=C(\mathcal{M})>0$ such that
$$\int_{\mathcal{M}}\frac{|w|^2}{\rho^4} dv_g \leq C\int_{\mathcal{M}} |\Delta w|^2 dv_g$$
for any $w\in C^{\infty}_0(\mathcal{M}, \mathbf{R})$.
\end{thm}
\begin{proof}
Let $\psi>0$ be the eigenfunction associated to the first eigenvalue $\lambda_1>0$ of $-\Delta_g$ on $\mathcal{M}$ with $\psi|_{\partial \mathcal{M}}=0$. Then we have $\psi\in C^\infty(\overline{\mathcal{M}})$ and $|\nabla \psi|> 0$ on $\partial {\mathcal{M}}$ by the Hopf boundary point Lemma (see e.g. \cite[Theorem E.4]{MR2365237}), noting that $\frac{\partial \psi}{\partial \nu}<0$ and the tangential derivative of $\psi$ vanishes on the boundary. Now as in the proof of \cite[Theorem 6.2]{DD14}, for fixed $\gamma>0$, define $$\mathcal{M}_\gamma = \{x\in \mathcal{M}: \rho(x)\geq \gamma\}$$ and $$\mathcal{M}^\gamma = \{x\in \mathcal{M}: \rho(x) < \gamma\}.$$ Then there exist constants $\epsilon, L, \delta>0$ depending on $\mathcal{M}$ and a smooth cut-off function $\chi \in C^{\infty}(\overline{\mathcal{M}}, \mathbf{R})$ such that
$$
\psi(x)\leq L\rho(x), \quad x\in \mathcal{M} \quad\text{and}\quad |\nabla \psi|(x)\geq \delta,\quad x\in \mathcal{M}^{2\epsilon}\,,
$$
$\chi =1$ on $\overline{\mathcal{M}^\epsilon}$, $\chi=0$ on $M \setminus  \mathcal{M}^{2\epsilon}$, and $|\nabla^i \chi | \leq C\epsilon^{-i}, i=1,2$ on $\mathcal{M}^{2\epsilon}  \setminus \mathcal{M}^{\epsilon}$. For any $w\in C^{\infty}_0(\mathcal{M}, \mathbf{R})$, denote $$\tilde{w} = w\cdot \chi \in C^{\infty}_0(\mathcal{M}^{2\epsilon},\mathbf{R}).$$ Then by Theorem \ref{secondhardy-v2} (replacing $\mathcal{M}$ by  $\mathcal{M}^{2\epsilon}$, $w$ by $\tilde{w}$ and $\rho$ by $\psi$) we get
\begin{align}\label{temp1}
\int_{\mathcal{M^\epsilon}} \frac{|w|^2}{\rho^4} dv_g &\leq \frac{L^4}{\delta^2}\int_{\mathcal{M}^{2\epsilon}} \frac{|\tilde{w}|^2|\nabla \psi|^2}{\psi^4}dv_g \notag\\
&\leq \frac{C(\mathcal{M})(1+\delta^{-2})^2L^4}{\delta^2}\int_{\mathcal{M}^{2\epsilon}}|\Delta \tilde{w}|^2 dv_g
 \notag\\
& \leq \frac{C(\mathcal{M},\epsilon)(1+\delta^{-2})^2 L^4}{\delta^2}\int_{\mathcal{M}}|\Delta w|^2dv_g\,.
\end{align}
Now let $m_\epsilon = \min_{\mathcal{M}_\epsilon} \psi>0, D = \max_{\mathcal{M}} \psi>0$, then by \cite[Theorem 6.3]{DD14} (with $p=2, s=\frac{1}{2}$) we have
\begin{align}\label{temp2}
\int_{\mathcal{M_\epsilon}} \frac{|w|^2}{\rho^4} dv_g &\leq m_\epsilon^{-\frac{1}{2}}\epsilon^{-4}\int_{\mathcal{M}} |w|^2 \psi^{\frac{1}{2}} dv_g\notag \\
&\leq 8m_\epsilon^{-\frac{1}{2}}\epsilon^{-4}\lambda_1^{-1} D^{\frac{1}{2}}\int_{\mathcal{M}} |\nabla w|^2 dv_g\notag
 \\
& \leq C(\mathcal{M},\epsilon, m_\epsilon, \lambda_1, D) \int_{\mathcal{M}}|\Delta w|^2dv_g\,.
\end{align}
Combining \eqref{temp1} and \eqref{temp2} completes the proof.
\end{proof}

As an immediate corollary of Theorem \ref{convexity} we get the following uniqueness result for weakly conformal-harmonic maps from a locally conformally flat $4$-manifold $\mathcal{M}$ into spheres.

\begin{cor}
Let $(\mathcal{M},g)$ be a four dimensional compact locally conformally flat Riemannian manifold with a smooth boundary $\partial \mathcal{M}$. Then there exists $\epsilon_0>0$ depending only on $\mathcal{M}$ such that for any weakly conformal-harmonic maps $u,v\in W^{2,2}(\mathcal{M},\mathbf{S}^n)$ with $u=v$,  $\partial_\nu u=\partial_\nu v$ on $\partial \mathcal{M}$ and
\begin{equation*}
\int_{\mathcal{M}}|\Delta u|^2 dv_g\leq \epsilon_0, \quad \int_{\mathcal{M}}|\Delta v|^2 dv_g\leq \epsilon_0,
\end{equation*}
we have $u\equiv v$ on $\mathcal{M}$.
\end{cor}

\subsection*{Acknowledgement} We would like to thank the referees for their valuable comments, which have significantly improved the presentation of the results in this paper. The first author acknowledges partial support from a COR research funding at UC Santa Cruz. The second author would like to thank the support by the National Natural Science Foundation of China (Grant No. 12201440) and the Fundamental Research Funds for the Central Universities (Grant No. YJ2021136).

\section{Preliminary}
In this section, we will fix some notations and recall a technical theorem ($\epsilon$-regularity) that will be used later. Throughout this section, $(\mathcal{M},g)$ denotes a smooth $4$-dimensional compact Riemannian manifold with a smooth boundary $\partial \mathcal{M}$ and $(\mathcal{N},h)$ is an $n$-dimensional smooth closed Riemannian manifold which can be embedded into $\mathbf{R}^K$. As mentioned in the introduction, a {\textit{weakly conformal-harmonic map}} $u$ from $\mathcal{M}$ into $\mathcal{N}\hookrightarrow \R^k$ is a map in $W^{2,2}(\mathcal{M}, \R^k)$ that is a critical point of the conformally invariant energy $\mathcal{E}(\cdot, \mathcal{M})$ defined in \eqref{energy} and takes values almost everywhere in $\mathcal{N}$. 

Note that the dimension $4$ is \textit{critical} for the analysis of weakly conformal-harmonic maps (e.g. a $W^{2,2}$ map falls in $L^p$ for any $p<\infty$ but barely fails to be continuous in dimension $4$). Now let $\Pi : \mathcal{N}_\delta \rightarrow \mathcal{N}$ be the nearest point projection map, which is well defined and smooth for $\delta>0$ small enough. Here $\mathcal{N}_\delta=\{y\in\R^k\,\vert\, \text{dist}\,(y, \mathcal{N})\leq \delta\}$. For $y\in \mathcal{N}$, let 
$$P(y)\equiv D \Pi (y):\R^k \rightarrow T_y \mathcal{N}$$
be the orthogonal projection onto the tangent plane $T_y \mathcal{N}$, and 
$$P^\bot(y)\equiv \text{Id}-D\Pi (y):\R^k \rightarrow (T_y \mathcal{N})^\bot\,,$$
where $D$ is the derivative with respect to the standard coordinates of $\mathbf{R}^K$. In the following, we will write $P$ (resp. $P^\bot$) instead of  $P(y)$ (resp. $P^\bot(y)$) and we will identify these linear transformations with their matrix representations in $M_n$. We also note that these projections are in $W^{2,2}(\mathcal{M},M_n)$ as soon as $u$ is in  $W^{2,2}(\mathcal{M},\mathcal{N})$ . Finally, note that the second fundamental form $A(\,\cdot\,)(\, \cdot\, , \,\cdot\,)$ of $\mathcal{N}\subset\R^k$ is defined by
$$A(y)(Y,Z):=D_Y P^\bot (y) (Z), \quad \forall \,y \in \mathcal{N} \quad \text{and} \quad Y , Z \in T_y \mathcal{N}.$$

We know that $u=(u^1, ..., u^k)\in W^{2,2}(\mathcal{M},\mathcal{N})$ is an \textit{intrinsic bi-harmonic map} if it satisfies the fourth order PDE (see \cite{Wang04-2} and \cite{LaR} for details)
\begin{align}\label{ibe}
\Delta^2u  = &-\Delta(\nabla P^\bot\nabla u)-{\rm div}(\nabla P^\bot\Delta u)+2\nabla P^\bot\nabla(\nabla P^\bot\nabla u)\notag\\
&+2\nabla P^\bot\nabla P^\bot\Delta u-(\nabla PP^\bot-P^\bot\nabla P)\nabla\Delta u\notag\\
&+P\langle\nabla P^\bot\nabla u,{D}_u({D}_u P^\bot)\nabla u\nabla u\rangle-2{\rm div}\langle\nabla P^\bot\nabla u,\nabla P^\bot P\rangle\\
&+2\langle\nabla P^\bot\nabla u,\nabla P^\bot\nabla P\rangle\,.\notag
\end{align}
Here $D_u P^\bot = D_{y} P^\bot (y)|_{y=u}$. Note that
\begin{equation}
\nabla P^\perp=D_u P^\perp\nabla u\,,
\end{equation}
and the following two terms in \eqref{ibe} are equivalent to:
\begin{equation}
\nabla P^\perp P = {D}_u P^\perp\nabla u P = A(u)(\nabla u,P)
\end{equation}
and
\begin{equation}
\nabla P^\perp \nabla P={D}_u P^\perp\nabla u \nabla P=A(u)(\nabla u,\nabla P).
\end{equation}

When $\mathcal{N}=\mathbf{S}^n$, we have (note that $\Delta u=\tau(u)-A(u)(\nabla u,\nabla u)$).
\begin{equation}\label{decomposition laplace}
P^\perp(\Delta u)=-\nabla P^\perp\nabla u=-u|\nabla u|^2\,,
\end{equation}
and therefore
\begin{equation}\label{2nd ff}
A(u)(\nabla u,\nabla u)=\nabla P^\perp\nabla u=u|\nabla u|^2.
\end{equation}
In particular, when $\mathcal{N} = \mathbf{S}^n$, the intrinsic bi-harmonic map equation can be rewritten as (see e.g. Lamm-Rivi\`{e}re \cite{LR})
\begin{equation}\label{ibe02}
\Delta^2 u \,=\, \Delta (V\cdot \nabla u) + \text{div }(w \nabla u) + W\cdot \nabla u\,,
\end{equation}
where
\beq
\label{ibe03}
\left\{
\begin{aligned}
V^{ij} &=  u^i\nabla u^j - u^j\nabla u^i\\
w^{ij}&= \text{div}\left(V^{ij}\right) \\
W^{ij}&= \nabla w^{ij}+ 2 \left[ \Delta u^i \nabla u^j - \Delta u^j \nabla u^i + |\nabla u|^2(u^i\nabla u^j - u^j\nabla u^i) \right]\,.
\end{aligned}
\right.
\eeq

\begin{rmk}
In local coordinates, the terms in \eqref{ibe} read as
\begin{align*}
&P\left(\langle  \nabla P^\bot  \nabla u , D_u(D_u P^\bot)\nabla u  \nabla u \rangle \right)\\ &=\sum_{\alpha, \beta,\gamma, i, j, k, m} P_{lk} \nabla_\alpha (P^\bot)_{ij} \nabla_{\alpha} u^j  D_{u^k} D_{u^\beta}(P^\bot)_{im} \nabla_\gamma u^\beta \nabla_\gamma u^m;
\end{align*}
\begin{align*}
 \langle \nabla P^\bot  \nabla u,  \nabla P^\bot  \nabla P \rangle\,&=\, \sum_{\alpha, \beta, i, j, k} \nabla_\alpha (P^\bot)_{ij} \nabla_{\alpha} u^j  \nabla_\beta (P^\bot)_{ik}    \nabla_\beta P_{kl}\,;\\
\text{div} \,\langle \nabla P^\bot  \nabla u,  \nabla P^\bot  P \rangle \,&=\,  \sum_{\alpha, \beta, i, j, k} \nabla_\beta \left [\nabla_\alpha (P^\bot)_{ij} \nabla_{\alpha} u^j  \nabla_\beta (P^\bot)_{ik}    P_{kl}\right]\,,
\end{align*}
where $\nabla_\alpha:= \nabla_{e_\alpha}$ and $\{e_\alpha\}_{\alpha=1,\cdots,4}$ is a local orthonormal frame on $\mathcal{M}$.
\end{rmk}

To end this section, let us recall the following version of $\varepsilon$-regularity for approximate intrinsic and extrinsic bi-harmonic maps into spheres, which will be useful later. Throughout the rest of this paper, $\bar\nabla$ and $\bar\Delta$ will denote the connection and Laplacian with respect to the Euclidean metric.

\begin{thm}(\cite[Theorem A.4]{LL21}, c.f. \cite{LR, LaR})\label{ereg}
Let $B_1\subset \mathbf{R}^4$ be the unit 4-ball. There exist $\eps>0$, $0<\delta<1$, $\alpha>0$ and $C>0$ independent of $u$ such that if $u\in W^{2,2}(B_1, \mathbf{S}^n)$ is a solution of 
\begin{equation}
\label{nf2}
\bar\Delta^2 u =\bar \Delta (V \bar\nabla u) + \text{div}( w \bar\nabla u) + \bar\nabla \omega \bar\nabla u + F\bar\nabla u ,
\end{equation}
where $V \in W^{1,2}(B_1, M_{n+1} \otimes \Lambda^{1} \R^4)$, $w\in L^2(B_1, M_{n+1})$, $\omega \in L^2(B_1, \mathrm{s}o_{n+1})$ and $F \in L^2\cdot W^{1,2}(B_1,  M_{n+1} \otimes \Lambda^{1} \R^4)$, which satisfy 
\beq
\label{cc}
\begin{split}
\vert V\vert  &\leq C\vert \bar\nabla u\vert\,,\\
\vert F \vert  &\leq C \vert\bar\nabla u \vert\left(\vert \bar\nabla^2 u\vert + \vert \bar\nabla u\vert^2 \right)\,,\\
\vert w \vert + \vert \omega \vert  &\leq C\left( \vert \bar\nabla^2 u\vert + \vert \bar\nabla u\vert^2 \right ) 
\end{split}
\eeq
almost everywhere (where $C>0$ is a constant depending only on $\mathcal{N}$) and
\begin{equation}\label{smallbien01}
\Vert \bar\Delta u\Vert_{L^2(B_1)} \leq \eps,
\end{equation}
then we have $u\in W_{loc}^{3, 4/3} (B_1, \R^{n+1})$ and
\begin{equation*}
\Vert \bar\nabla^3 u \Vert_{L^{\frac{4}{3}}(B(p, \rho))} +\Vert \bar\nabla^2 u \Vert_{L^2(B(p, \rho))} +\Vert \bar\nabla u \Vert_{L^4(B(p, \rho))}  \leq C \rho^\alpha \Vert \bar\Delta u\Vert_{L^2(B_1)}
\end{equation*}
for all $p\in B_{\frac{1}{4}}$ and $0\leq  \rho \leq \delta$. Moreover, $u\in W^{3,\infty} (B_{\frac{1}{16}}, \R^{n+1})$ and for $l = 1, 2, 3$ we have
\beq
\label{Fereg} \vert \bar\nabla^l u\vert(0) \leq C_{l}\Vert\bar \Delta u\Vert_{L^2(B_1)}
\eeq
for some constant $C_l>0$. In particular, by rescaling we have for $x\in B_1$ and $l=1,2,3$:
\beq
\label{PTWSest} 
\vert \bar\nabla^l u\vert(x) \leq C_{l}(1-|x|)^{-l} \Vert \bar\Delta u\Vert_{L^2(B_1)}\,.
\eeq
\end{thm}

\section{Second-order Hardy inequality on manifolds}\label{Hardyineq}

The first order Hardy inequality on Euclidean domains is well-known. Let's recall the second order Hardy inequality on Euclidean domains (see e.g. \cite{ER99, KM97}).

\begin{thm}(\cite[Theorem 2]{ER99})
Let $\Omega \subset \mathbf{R}^n$ be a bounded Lipschitz domain. There exists a constant $C>0$ depending only on $\Omega$ such that if $w\in W^{2,2}_0(\Omega,\mathbf{R}^K)$, then we have
\begin{equation}
  \int_{\Omega} |w|^2(\text{dist}(x, \partial \Omega))^{-4}dx\leq C\int_{\Omega}|\bar\Delta w|^2dx.
\end{equation}
\end{thm}
Recently, in \cite{DD14} D'Ambrosio and Dipierro extended the first order Hardy inequality to the Riemannian manifold setting (c.f. \cite{Car97, KO09}), more precisely, among other things they proved:

\begin{thm}{\cite[Theorem 2.1]{DD14}}\label{firtHardy}
Let $(\mathcal{M},g)$ be an $n$ dimensional compact Riemannian manifold with a smooth boundary $\partial \mathcal{M}$. Let $\rho \in W^{1,2}_{loc}(\mathcal{M},\mathbf{R})$ be a nonnegative function such that $\Delta \rho \leq 0$ on $\mathcal{M}$ in the weak sense, then we have $\frac{|\nabla \rho|^2}{|\rho|^2} \in L^1_{loc}(\mathcal{M}, \mathbf{R})$ and
$$
\int_{\mathcal{M}} \frac{|w|^2}{\rho^2}|\nabla \rho|^2 dv_g \leq 4\int_{\mathcal{M}} |\nabla w|^2dv_g 
$$
for any $w\in C^{\infty}_0(\mathcal{M},\mathbf{R})$.
\end{thm}

Building on the work of D'Ambrosio and Dipierro \cite{DD14}, which utilizes techniques introduced by Mitidieri in \cite{Mit00}, we establish a general version of the second-order Hardy inequality applicable to any smooth manifold (Theorem \ref{secondhardy}). Since Theorem \ref{secondhardy} relies on Theorem \ref{secondhardy-v2}, which is valid for certain smooth manifolds, we first provide a proof of Theorem \ref{secondhardy-v2}.

\begin{proof} (of Theorem \ref{secondhardy-v2})
First note that by the assumption $\Delta \rho \leq 0$ we have in the weak sense that
$$
\Delta\left(\frac{1}{2\rho^2}\right)=\text{div}\left(-\frac{\nabla \rho}{\rho^3}\right)=-\frac{\Delta \rho}{\rho^3}+ \frac{3|\nabla \rho|^2}{\rho^4}\geq \frac{3|\nabla \rho|^2}{\rho^4}.
$$
Therefore, for any $w\in C^\infty_0(\mathcal{M}, \mathbf{R})$ we have
\begin{align}\label{hardyproof}
&3 \int_{\mathcal{M}}\frac{|w|^2|\nabla \rho|^2}{\rho^4} dv_g\leq \int_{\mathcal{M}}|w|^2\Delta\left(\frac{1}{2\rho^2}\right) dv_g\notag\\
=& -2\int_{\mathcal{M}} w\nabla w\cdot \nabla \left(\frac{1}{2\rho^2}\right) dv_g = 2\int_{\mathcal{M}} \left(|\nabla w|^2+ w\Delta w\right) \left(\frac{1}{2\rho^2}\right) dv_g \notag\\
\leq & \left( \int_{\mathcal{M}}\frac{3|w|^2|\nabla \rho|^2}{\rho^4}dv_g\right)^{1/2}\left( \int_{\mathcal{M}}\frac{|\Delta w|^2}{3|\nabla \rho|^2} dv_g\right)^{1/2}+ \int_{\mathcal{M}}\frac{|\nabla w|^2}{\rho^2} dv_g\,.
\end{align}
By applying the first-order Hardy inequality (Theorem \ref{firtHardy}) to $|\nabla w|$, and using the fact that $|\nabla \rho|\geq \delta$ a.e. on $\mathcal{M}$, along with Kato's inequality, Reilly’s formula, and the Poincaré inequality, we obtain:
$$
\int_{\mathcal{M}}\frac{|\nabla w|^2}{\rho^2} dv_g \leq \frac{1}{\delta^2}\int_{\mathcal{M}}\frac{|\nabla w|^2}{\rho^2}|\nabla \rho|^2 dv_g \leq \frac{4}{\delta^2}\int_{\mathcal{M}}|\nabla^2 w|^2 dv_g\leq \frac{C(\mathcal{M})}{\delta^2}\int_{\mathcal{M}}|\Delta w|^2dv_g\,.
$$
Inserting this back to \eqref{hardyproof} we get
\begin{align*}
\int_{\mathcal{M}}\frac{|w|^2|\nabla \rho|^2}{\rho^4} dv_g \leq C(\mathcal{M})(1+ \delta^{-2})^2\int_{\mathcal{M}} |\Delta w|^2 dv_g
\end{align*}
for some $C(\mathcal{M})>0$.
\end{proof}

\section{Proof of the main result}
In this section, we prove Theorem \ref{convexity}. Let us first fix some notations. Since $\mathcal{M}$ is compact and locally conformally flat with a smooth boundary $\partial \mathcal{M}$, we can choose a smooth atlas 
$$\{\Phi_{i}: B_{2r_i}(p_i)\to \mathbf{R}^4\}_{i=1}^N$$
for $\mathcal{M}$ 
such that $\{p_i\}_{i=1}^k\subset \mathcal{M}, \{p_i\}_{i=k+1}^N \subset \partial \mathcal{M}$ and $B_{2r_i}(p_i)$ is conformally flat for each $i$. We assume that
\beq \label{conformflat}
g=e^{2\phi_i}\bar{g} \quad \text{on } B_{2r_i}(p_i)\,,
\eeq
where $\bar{g}$ is the Euclidean metric. Here $B_{2r_i}(p_i) \subset \mathcal{M}$ denotes the set
$$
B_{2r_i}(p_i):= \{y\in \mathcal{M}\cup \partial \mathcal{M}: \text{dist}_g(y, p_i)< 2r_i\}, \quad i=1,\cdots, N.
$$
Moreover,
$$\mathcal{M} \cup \partial \mathcal{M} \text{ is covered by } \left\{B_{r_i}(p_i)\right\}_{i=1}^N $$
and every point in $\mathcal{M} \cup \partial \mathcal{M}$ is covered by $\left\{B_{2r_i}(p_i)\right\}_{i=1}^N$ at most $L$ times, see e.g. \cite[Lemma 3.3]{Str85}.
For $i=1,\cdots, N$, let 
$$U_i:= \Phi_i(B_{r_i}(p_i))\quad \text{and}\quad 2U_i:= \Phi_i(B_{2r_i}(p_i))\,.$$
When choosing the smooth atlas, we additionally require that there exists $\delta>0$ depending only on $\mathcal{M}$ and $\partial \mathcal{M}$ such that for any $i=k+1,\cdots,N$ and any point $x\in \{y\in \mathcal{M}: \text{dist}(y, \partial \mathcal{M}) \leq \delta\} \cap B_{r_i}(p_i)$ we have
\begin{equation}\label{interiorball}
\Phi_i(B_{\rho(x)}(x))\subset U_j \quad \text{for some } j=k+1,\cdots, N\,,    
\end{equation}
where $\rho(x)=\text{dist}_g(x,\partial \mathcal{M})${\footnote{To see this, let $\Sigma_i=\partial B_{r_i}(p_i)\cap\partial \mathcal{M}, i=k+1,\cdots, N$ be the projection of $ B_{r_i}(p_i)$ onto $\partial M$ by normal geodesics. $\Gamma_i=\partial \Sigma_i$ denotes its boundary in $\partial \mathcal{M}$. We choose the smooth atlas of $\mathcal{M}$ in such a way that every $\Sigma_i$ is covered by finitely many $\Sigma_{j_\alpha}$'s and all the intersecting points $q_{st}=\Gamma_{s}\cap\Gamma_t\cap\Sigma_i$ stay in the interior of $\Sigma_{l}$, where $\alpha=1,\dots,\alpha_i$, $j_\alpha\in [k+1, N]$, $l\neq s,t$ and $s,t,l=i,j_1,\dots,j_{\alpha_i}$. Then there exists a constant $C_i>0$  such that for any $y_0\in \Sigma_i$, there holds $\text{dist}(y_0,\Gamma_l)\geq C_i$ for some $l=i,j_1,\dots,j_{k_i}$. Now, for any $x\in \{y\in \mathcal{M}: \text{dist}_g(y,\partial \mathcal{M})\leq \delta\}$, there exists $x_0\in \partial \mathcal{M}$ such that $x_0\in \Sigma_i$ for some $i\in [k+1,N]$ and $\rho(x)=\text{dist}_g(x,\partial \mathcal{M})=\text{dist}_g(x,x_0)$. Since $B_{\rho(x)}(x)\subset B_{\delta}(x)\subset B_{2\delta}(x_0)$, we have $B_{\rho(x)}(x)\subset B_{r_i}(p_i)$ for some $i=k+1,\dots, N$ if we choose $2\delta<\min\limits_{i=k+1,\dots,N}\{C_i\}$.}}. Moreover, 
\begin{equation}\label{boundarynbhd}
    \left\{y\in \mathcal{M}: \text{dist}(y, \partial \mathcal{M}) \geq \delta\right\} \subset \bigcup_{i=1,\cdots, k} B_{r_i}(p_i)\,.
\end{equation}

Let $\Omega_1:=B_{r_1}(p_1)$ and
\beq\label{omegaset}
\Omega_{i+1}:= \left(\bigcup_{j=1,\cdots,i+1}B_{r_j}(p_j)\right)\setminus \left(\bigcup_{j=1,\cdots, i}\Omega_{j}\right), \quad i=1,\cdots, N-1
\eeq
be a disjoint partition of $\mathcal{M}$. Now for any $u\in W^{2,2}(\mathcal{M}, \mathbf{S}^n)$ we define
\begin{equation}\label{ubar}
\bar{u}_i(x):= u\left( \Phi_i^{-1}(x)\right)\,, \quad x\in 2U_i,\quad i=1,\cdots, N\,.
\end{equation}
\begin{rmk}\label{local-biharmonic}
Note that by Proposition \ref{local property} we know that if $u$ is a conformal-harmonic map on $\mathcal{M}$, then $\bar{u}_i$ defined in \eqref{ubar} is an intrinsic biharmonic map on $2U_i, i=1,\cdots, N$. 
\end{rmk}

In what follows, we will denote $|\cdot|$ the norm with respect to the Euclidean metric in $\mathbf{R}^K$. As before, $\bar\nabla$ and $\bar\Delta$ will denote the connection and Laplacian with respect to the Euclidean metric $\bar g$, but note that on each $\Omega_i, i=1,\cdots, N$ the information of $\Phi_i$ is embedded in these two operators and we do not differentiate the notations.

\begin{lem}\label{lhs}
There exists $\epsilon_0>0$ depending only on $\mathcal{M}$ such that if $u,v$ are as in Theorem \ref{convexity}, then we have
\begin{equation}\label{tau laplace}
\int_{ \mathcal{M}}|\bar\Delta(v-u)|^2 dv_{\bar{g}}\leq 4\int_{ \mathcal{M}}|\bar\tau(v)-\bar\tau(u)|^2 dv_{\bar{g}}\,,
\end{equation}
where $\bar\tau(u) = (\bar\Delta u)^T$ is the tension field of $u$ with respect to the flat connection.
\end{lem}

\begin{proof}
Let $\bar{u}_i, \bar{v}_i$ be defined as in \eqref{ubar}, then use the conformal change of Laplacian (c.f. \eqref{conformalchange-1}), for $i=1,\cdots, N$ we get
\begin{align}\label{laplacecom-1}
  \int_{2U_i}|\bar{\Delta} \bar u_i|^2dx 
  &=  \int_{B_{2r_i}(p_i)}|\bar{\Delta} u|^2 dv_{\bar{g}}\notag\\
  &\leq  C_i\int_{B_{2r_i}(p_i)}(|\Delta u|^2+|\nabla u|^2)dv_g\leq C_i\sqrt{\epsilon_0}
\end{align}
and
\beq\label{laplacecom-2}
\int_{2U_i}|\bar{\nabla} \bar{v}_i|^4dx = \int_{B_{2r_i}(p_i)}|\bar{\nabla} v|^4 dv_{\bar{g}}=\int_{B_{2r_i}(p_i)}|\nabla v|^4dv_g \leq \epsilon_0\,,
\eeq
where $C_i=C_i(\Phi_i, \phi_i)>0$ are positive constants. Moreover, using \eqref{laplacecom-1}, \eqref{laplacecom-2}, \eqref{bienergy condition} and  
$$\int_{\mathcal{M}}|\Delta u|^2dv_g =  \int_{\mathcal{M}}|\tau(u)|^2 dv_g+\int_{\mathcal{M}}|\nabla u|^4 dv_g \,,$$ 
we have
\beq \label{smallenergy-1}
\int_{\mathcal{M}} |\bar\Delta u|^2 dv_{\bar g} \leq C\int_{\mathcal{M}} \left(|\Delta u|^2 + |\nabla u|^2\right)dv_{g} \leq C\sqrt{\epsilon_0}
\eeq
and
\beq \label{smallenergy-2}
\int_{\mathcal{M}} |\bar\nabla v|^4 dv_{\bar g} = \int_{\mathcal{M}} |\nabla v|^4dv_{g} \leq \epsilon_0\,,
\eeq
where $C>0$ may depend on all $\phi_i$ and $\Phi_i, i=1,\cdots, N$. Now using the decomposition $\bar\Delta u = \bar\tau(u)-u|\bar\nabla u|^2$, we have
    \begin{align*}
       &\int_{\mathcal{M}}|\bar\Delta( v- u)|^2dv_{\bar g}=\int_{ \mathcal{M}}| \bar\tau(  v)- \bar\tau(  u)-  v| \bar\nabla   v|^2+  u| \bar\nabla   u|^2|^2 dv_{\bar g}\\
       \leq &2\int_{ \mathcal{M}}| \bar\tau(  v)- \bar\tau(  u)|^2 dv_{\bar g} +2\int_{ \mathcal{M}}\left|  v| \bar\nabla   v|^2-  u| \bar\nabla   u|^2\right|^2 dv_{\bar g}\\
       =&2\int_{ \mathcal{M}}| \bar\tau(  v)- \bar\tau(  u)|^2 dv_{\bar g} +2\int_{ \mathcal{M}}\left|  v\left(| \bar\nabla   v|^2-| \bar\nabla   u|^2\right)+(  v-  u)| \bar\nabla   u|^2\right|^2 dv_{\bar g}\\
       \leq&2\int_{ \mathcal{M}}| \bar\tau(  v)- \bar\tau(  u)|^2 dv_{\bar g}+4\int_{ \mathcal{M}}\left|| \bar\nabla   v|^2-| \bar\nabla   u|^2\right|^2dv_{\bar g} +4\int_{ \mathcal{M}}|  v-  u|^2| \bar\nabla   u|^4dv_{\bar g}\\
       \leq&2\int_{ \mathcal{M}}| \bar\tau(  v)- \bar\tau(  u)|^2 dv_{\bar g} +4\left(\int_{ \mathcal{M}}| \bar\nabla(  v+  u)|^4dv_{\bar g}\right)^{\frac12}\left(\int_{ \mathcal{M}}| \bar\nabla(  v-  u)|^4dv_{\bar g}\right)^{\frac12}dv_{\bar g}\\
       &+4\int_{ \mathcal{M}}|  v-  u|^2|\bar \nabla   u|^4dv_{\bar g}\\
       \leq&2\int_{ \mathcal{M}}| \bar\tau(  v)- \bar\tau(  u)|^2dv_{\bar g} +C\sqrt{\epsilon_0}\int_{ \mathcal{M}}| \bar\Delta (  v-  u)|^2dv_{\bar g}+4\int_{ \mathcal{M}}|  v-  u|^2| \bar\nabla   u|^4dv_{\bar g}\,,
   \end{align*}
where we have used \eqref{bienergy condition}, \eqref{smallenergy-1}, \eqref{smallenergy-2},
\begin{equation}\label{smallenergy-3}
     \int_{\mathcal{M}}|\bar \nabla u|^4dv_{\bar g}\leq\int_{\mathcal{M}}|\bar \Delta u|^2dv_{\bar g} \leq C\sqrt{\epsilon_0}
\end{equation}
and
\begin{equation}\label{control-1}
    \left(\int_{\mathcal{M}}|\bar\nabla(v-u)|^4dv_{\bar g}\right)^{\frac14}\leq C\left(\int_{\mathcal{M}}|\bar\Delta (v-u)|^2dv_{\bar g}\right)^{\frac12}\,,
    \end{equation}
where $C>0$ is a universal constant.
For the last term above, we claim that
\begin{equation}\label{claim}
    \int_{\mathcal{M}}|v-u|^2|\bar\nabla u|^4dv_{\bar g}\leq C_0{\epsilon_0}\int_{\mathcal{M}}|v-u|^2\rho^{-4}dv_{g}\,,
\end{equation}
where $\rho(x)=\text{dist}_g(x,\partial\mathcal{M})$ and $C_0=C_0(\mathcal{M})>0$. If the claim \eqref{claim} is true, then the second Hardy inequality (Theorem \ref{secondhardy}) imply that (since $v-u\in W^{2,2}_0(\mathcal{M}, \mathbf{S}^n)$ and $\frac{\partial(v-u)}{\partial \nu}=0$ on $\partial \mathcal{M}$)
\begin{align}\label{claim-2}
   &\int_{\mathcal{M}}|v-u|^2|\bar\nabla u|^4 dv_{\bar g}
   \leq C{\epsilon_0}\int_{\mathcal{M}}|\Delta (v-u)|^2dv_{g}\notag\\
   = &C{\epsilon_0}\sum_{i=1,\cdots, N}\int_{\Omega_i}|\Delta (v-u)|^2dv_{g}\notag\\
   \leq &C{\epsilon_0}\sum_{i=1,\cdots, N}\int_{\Omega_i}\left(|\bar\Delta (v-u)|^2 +|\bar\nabla (v-u)|^2 \right)dv_{\bar g} \notag\\
   =&C{\epsilon_0} \int_{\mathcal{M}}\left(|\bar\Delta (v-u)|^2 +|\bar\nabla (v-u)|^2 \right)dv_{\bar g} \leq C{\epsilon_0} \int_{\mathcal{M}}|\bar\Delta (v-u)|^2 dv_{\bar g}\,.
\end{align}
where we used
$\int_{\Omega_i}|\Delta (v-u)|^2dv_{g} \leq C_i\int_{\Omega_i}\left(|\bar\Delta (v-u)|^2+|\bar\nabla (v-u)|^2 \right)dv_{\bar g}$ for each $i$, which is similar to \eqref{laplacecom-1}. Therefore, we obtain \eqref{tau laplace} by choosing $\epsilon_0$ sufficiently small. To complete the proof of this lemma, it remains to show that the claim \eqref{claim} holds.

Using \eqref{laplacecom-1}, Remark \ref{local-biharmonic} and Theorem \ref{ereg}, and choosing $\epsilon_0$ small enough we have 
\begin{equation}
    |\bar{\nabla}^l \bar{u}_i|(x)\leq C(\text{dist}_{\bar g}(x, \partial (2U_i)))^{-l}{\epsilon_0}^{1/4}\,,
\end{equation}
for any $x\in 2U_i, i=1,\cdots, k$ and $l=1,2,3$, where $C>0$ depends on the $C_i$ above.
In particular, define
$$
r_0 := \min_{i=1,\cdots, k}\text{dist}_{\bar g}(U_i, \partial(2U_i))\,,
$$
then we have
\begin{equation}
    |\bar{\nabla}^l \bar{u}_i|(x)\leq Cr_0^{-l}{\epsilon_0}^{1/4} \quad\text{for any } x\in U_i, i=1,\cdots, k \text{ and } l=1,2,3,
\end{equation}
\begin{equation}\label{gradientbound1}
    |\bar{\nabla}^l u|(x)\leq Cr_0^{-l}{\epsilon_0}^{1/4} \quad\text{for any } x\in B_{r_i}(p_i), i=1,\cdots, k \text{ and } l=1,2,3,
\end{equation}
and therefore
\begin{equation}\label{gradientbound2}
    |\bar \nabla^l u|(x)\leq Cr_0^{-l}{\epsilon_0}^{1/4}(\text{diam}_g(\mathcal{M}))^l\rho^{-l}(x) \leq C_1{\epsilon_0}^{1/4} \rho^{-l}(x)
\end{equation}
for any $x\in B_{r_i}(p_i), i=1,\cdots, k$ and $l=1,2,3$, where the $C>0$ in \eqref{gradientbound1} and \eqref{gradientbound2} may depend on $\Phi_i, i=1,\cdots, k,$ and 
$$C_1: =\max_{l=1,2,3}\left\{ Cr_0^{-l}(\text{diam}_g(\mathcal{M}))^l\right\}\,.$$

Now by the choice of the smooth atlas, in particular, \eqref{interiorball}, with a similar argument as above, we have
\begin{equation}\label{gradientbound3}
    |\bar \nabla^l u|(x)\leq  C_2{\epsilon_0}^{1/4} \rho^{-l}(x)
\end{equation}
for any $x\in B_{r_i}(p_i)\cap \{y\in \mathcal{M}: \text{dist}(y, \partial \mathcal{M}) \leq \delta\}, i=k+1,\cdots, N$ and $l=1,2,3$. Here $C_2>0$ may depend on all $\Phi_i$ and $\phi_i, i=k,\cdots, N.$ Then with a argument similar to \eqref{claim-2} we get \eqref{claim}. This completes the proof of the lemma.
\end{proof}

\begin{lem}\label{key estimate2}
There exists $\epsilon_0>0$ depending only on $\mathcal{M}$ such that if $u,v$ are as in Theorem \ref{convexity}, then we have
\begin{align}\label{difference laplace}
  &\int_{\mathcal{M}}|\bar\Delta v|^2dv_{\bar g}-\int_{\mathcal{M}}|\bar \Delta u|^2dv_{\bar g}-\int_{\mathcal{M}}|\bar \Delta (v-u)|^2dv_{\bar g}\notag\\
  \geq&-C{\epsilon_0}^{1/4}\int_{\mathcal{M}}|\bar \Delta (v-u)|^2dv_{\bar g}+4\int_{\mathcal{M}}|\bar \nabla u|^2\bar \nabla u \cdot\bar \nabla(v-u)dv_{\bar g}\,.
\end{align}
\end{lem}

\begin{proof}

By Remark \ref{local-biharmonic}, we know that $u$ is an intrinsic biharmonic map on each $\Omega_i, i=1,\cdots, N,$ with respect to the flat connection $\bar \nabla$, namely, we have
\begin{align*}
\bar\Delta^2u  = &P(\bar\Delta^2 u)+P\langle\bar\nabla P^\bot\bar\nabla u,{D}_u({D}_u P^\bot)\bar\nabla u\bar\nabla u\rangle \notag\\
&-2{\rm div}_{\bar g}\langle\bar\nabla P^\bot\bar\nabla u,\bar\nabla P^\bot P\rangle+2\langle\bar\nabla P^\bot\bar\nabla u,\bar\nabla P^\bot\bar\nabla P\rangle\,.
\end{align*}

Since $u=v$, $\partial_\nu u=\partial_\nu v$ on $\partial \mathcal{M}$, we have
\begin{align*}
   & \int_{\mathcal{M}}|\bar\Delta v|^2dv_{\bar g}-\int_{\mathcal{M}}|\bar\Delta u|^2dv_{\bar g}-\int_{\mathcal{M}}|\bar\Delta (v-u)|^2dv_{\bar g}\\
   =&2\int_{\mathcal{M}}\left\langle \bar\Delta^2u,v-u\right\rangle dv_{\bar g}\\
   =&2\int_{\mathcal{M}}\left\langle P^\bot(\bar \Delta^2u),v-u\right\rangle dv_{\bar g}+2\int_{\mathcal{M}}\left\langle  P(\langle\bar\nabla P^\bot\bar\nabla u,{D}_u({D}_u P^\bot)\bar\nabla u \bar\nabla u\rangle),v-u\right\rangle dv_{\bar g}\\
&-4\int_{\mathcal{M}}\left\langle{\rm div}_{\bar g}\langle\bar\nabla P^\bot\bar\nabla u,\bar\nabla P^\bot P\rangle,v-u\right\rangle dv_{\bar g}+4\int_{\mathcal{M}}\left\langle\langle\bar\nabla P^\bot\bar\nabla u,\bar\nabla P^\bot\bar\nabla P\rangle,v-u\right\rangle  dv_{\bar g}\\
=&2\int_{\mathcal{M}}\left\langle P^\bot(\bar\Delta^2u),v-u\right\rangle dv_{\bar g} +2\int_{\mathcal{M}}\left\langle\langle \bar\nabla P^\bot\bar\nabla u,{D}_u({D}_u P^\bot)\bar\nabla u\bar\nabla u\rangle,P(v-u)\right\rangle dv_{\bar g}\\
&+4\int_{\mathcal{M}}\left\langle\langle\bar\nabla P^\bot\bar\nabla u,\bar\nabla P^\bot P\rangle,\bar\nabla(v-u)\right\rangle dv_{\bar g}+4\int_{\mathcal{M}}\left\langle\langle\bar\nabla P^\bot\bar\nabla u,\bar\nabla P^\bot\bar\nabla P\rangle,v-u\right\rangle dv_{\bar g}\\
=:&\textbf{I}+\textbf{II}+\textbf{III}+\textbf{IV}.
\end{align*}

For term \textbf{I}, we can use \ref{secondhardy}, a similar argument as in the proof of \eqref{claim} and the second order Hardy inequality (Theorem \ref{secondhardy}) to get
\begin{align}
    2\int_{\mathcal{M}}\left\langle P^\bot(\bar\Delta^2u),v-u\right\rangle dv_{\bar g}
    &\geq-C\int_{\mathcal{M}}\left|(v-u)^\bot\right|\cdot\left|P^\bot(\Delta^2u)\right| dv_{\bar g}\notag\\
    &\geq-C{\epsilon_0}^{1/4}\int_{\mathcal{M}}|v-u|^2\rho(x)^{-4} dv_g\notag\\
    &\geq-C{\epsilon_0}^{1/4}\int_{\mathcal{M}}|\Delta(v-u)|^2 dv_g\notag\\
    &\geq -C{\epsilon_0}^{1/4}\int_{\mathcal{M}}|\bar \Delta(v-u)|^2 dv_{\bar g}\,.
\end{align}

For term \textbf{II}, we note that the integrand reads as
$$ \sum_{\alpha, \beta,\gamma, i, j, k, m,l} P_{lk} \bar\nabla_\alpha (P^\bot)_{ij} \bar\nabla_{\alpha} u^j  D_{u^k} D_{u^\beta}(P^\bot)_{im}  \bar\nabla_\gamma u^\beta \bar\nabla_\gamma u^m (v-u)^l\,.$$
Now since the target manifold is $\mathbf{S}^n$, we know that $u$ is the unit normal vector at the point $u\in  \mathbf{S}^n$ and $P^\bot({\mathbf{v}}) = \langle {\mathbf{v}}, u\rangle u$ for any vector ${\mathbf{v}} \in T_u(\mathbf{R}^{n+1})$ so that
$$
P^\bot_{ij} = u^i u^j \quad\text{and}\quad P^\bot (\bar\Delta u) = - \bar\nabla P^\bot \bar\nabla u = - u |\bar\nabla u|^2\,.
$$
Therefore, in this case the integrand in term \textbf{II} becomes
\beq
\label{termII}
\sum_{\beta,\gamma, i, k, m,l} P_{lk}   u^i|\nabla u|^2 (\delta_{i\beta}\delta_{mk} + \delta_{ik}\delta_{m\beta}) \bar\nabla_\gamma u^\beta \bar\nabla_\gamma u^m (v-u)^l \,=\,0\,.
\eeq

For term \textbf{III}, we have (using again $P^\bot_{ij} = u^i u^j$)
\beq \label{termIII}
 \left\langle \langle\bar\nabla P^\bot\bar\nabla u,  \bar\nabla P^\bot  P \rangle, \bar\nabla(v-u)\right\rangle
=|\bar\nabla u|^2\bar\nabla u \cdot \bar\nabla (v-u).
\eeq

For term \textbf{IV}, we use again $P^\bot_{ij} = u^i u^j, P + P^\bot = \text{Id}$ and also $u\cdot \bar\nabla u =0$ on $\mathbf{S}^n$ to get
\begin{align*}
 &\int_{\mathcal{M}}\left\langle\langle \bar\nabla P^\bot \bar\nabla u, \bar\nabla P^\bot  \bar\nabla P\rangle , v-u\right\rangle dv_{\bar g}= \int_{\mathcal{M}} \bar\nabla P^\bot_{ij} \bar\nabla u^j\bar\nabla P^\bot_{ik}\bar\nabla P_{ks}(v-u)^sdv_{\bar g}
\\
=&-\int_{\mathcal{M}} |\bar\nabla u|^4 u \cdot (v-u) dv_{\bar g}\geq -2\int_{\mathcal{M}} |\bar\nabla u|^4 |(v-u)^\bot| dv_{\bar g}\\
\geq& -C\int_{\mathcal{M}} |v-u|^2|\bar\nabla u|^4dv_{\bar g} \,.
\end{align*}
Then by \eqref{claim} and \eqref{claim-2} we have
\beq \label{termIV}
\textbf{IV} \geq - C\epsilon_0^{1/4} \int_{\mathcal{M}}|\bar\Delta (v - u)|^2dv_{\bar g} \,.
\eeq
Now \eqref{difference laplace} follows directly by combining the estimates above.
\end{proof}

\begin{lem}\label{key estimate3}
There exists $\epsilon_0>0$ depending only on $\mathcal{M}$ such that if $u,v$ are as in Theorem \ref{convexity}, then we have
\begin{align}\label{difference tension}
 &\int_{\mathcal{M}}|\bar \tau(v)|^2dv_{\bar g}-\int_{\mathcal{M}}|\bar \tau(u)|^2dv_{\bar g}-\int_{\mathcal{M}}|\bar \tau(v)-\bar \tau(u)|^2dv_{\bar g}\notag \\
 \geq&-C{\epsilon_0}^{1/4}\int_\mathcal{M}|\bar \Delta(v-u)|^2dv_{\bar g}\,.
\end{align}
\end{lem}

\begin{proof}
It follows from Lemma \ref{key estimate2} and \eqref{difference laplace} that
\begin{align*}
\psi=&  \int_{\mathcal{M}} \vert \bar\tau(v)\vert^2  dv_{\bar g} -  \int_{\mathcal{M}} \vert\bar\tau(u)\vert^2  dv_{\bar g} - \int_{\mathcal{M}} \vert \bar\tau(v) - \bar\tau(u)\vert^2 dv_{\bar g}\\
=& \int_{\mathcal{M}} |\bar\Delta v|^2 dv_{\bar g} - \int_{\mathcal{M}} |\bar\Delta u|^2 dv_{\bar g} - \int_{\mathcal{M}} |\bar\Delta (v-u)|^2 dv_{\bar g}\\
&- \int_{\mathcal{M}} |\bar\nabla v|^4 - |\bar\nabla u|^4 dv_{\bar g} + \int_{\mathcal{M}} \left| v|\bar\nabla v|^2 + u|\bar\nabla u|^2\right|^2 dv_{\bar g}\\
&+ 2 \int_{\mathcal{M}} |\bar\nabla u|^2 u \bar\Delta v + |\bar\nabla v|^2 v\bar\Delta u\, dv_{\bar g}\\
\geq &- C{\epsilon_0}^{1/4} \int_{\mathcal{M}} |\bar\Delta (v-u)|^2 dv_{\bar g} + 4 \int_{\mathcal{M}} |\bar\nabla u|^2\bar\nabla u \bar\nabla (v-u) dv_{\bar g}\\
&- \int_{\mathcal{M}} |\bar\nabla v|^4 - |\bar\nabla u|^4 dv_{\bar g}
 + \int_{\mathcal{M}} \left| v|\bar\nabla v|^2 + u|\bar\nabla u|^2\right|^2 dv_{\bar g}\\
 &+ 2 \int_{\mathcal{M}} |\bar\nabla u|^2 u \bar\Delta v + |\bar\nabla v|^2 v\bar\Delta u\,  dv_{\bar g}.
\end{align*}
 Therefore,  using $u|_{\partial \mathcal{M}}=v|_{\partial \mathcal{M}}$ and $\partial_\nu u |_{\partial \mathcal{M}}=\partial_\nu v|_{\partial \mathcal{M}}$, we have

\begin{align*}
\psi\geq  &-C{\epsilon_0}^{1/4}\int_{\mathcal{M}} |\bar\Delta (v-u)|^2  dv_{\bar g} - 2 \int_{\mathcal{M}} \bar\nabla |\bar\nabla u|^2 u \,\bar\nabla v  dv_{\bar g} - 2 \int_{\mathcal{M}}  |\bar\nabla u|^2 u \bar\Delta v  \,dv_{\bar g} \\
& - 2 \int_{\mathcal{M}} \bar\nabla |\bar\nabla u|^2 v\bar \nabla u  \,dv_{\bar g} - 2 \int_{\mathcal{M}} |\bar\nabla u|^2 v \bar\Delta u  \,dv_{\bar g} - 4 \int_{\mathcal{M}}|\bar\nabla u|^4 dv_{\bar g} \\
& - \int_{\mathcal{M}}(|\bar\nabla v|^4 - |\bar\nabla u|^4) dv_{\bar g} + \int_{\mathcal{M}} \left| v|\bar\nabla v|^2 + u|\bar\nabla u|^2\right|^2  dv_{\bar g}\\
&+ 2 \int_{\mathcal{M}} |\bar\nabla u|^2 u \bar\Delta v + |\bar\nabla v|^2 v\bar\Delta u\, dv_{\bar g}\\
=&- C{\epsilon_0}^{1/4} \int_{\mathcal{M}} |\bar\Delta (v-u)|^2  dv_{\bar g} + \int_{\mathcal{M}} \bar\nabla |\bar\nabla u|^2 \bar\nabla |v-u|^2 dv_{\bar g}\\
&+ 2 \int_{\mathcal{M}}  v\bar\Delta u (|\bar\nabla v|^2 - |\bar\nabla u|^2)  dv_{\bar g} + 2\int_{\mathcal{M}} |\bar\nabla u|^2 uv (|\bar\nabla v|^2 - |\bar\nabla u|^2) dv_{\bar g}\\
&- \int_{\mathcal{M}} |v-u|^2|\bar\nabla u|^4 dv_{\bar g}\,,
\end{align*}
where we have used $|u|^2 = |v|^2 = 1$ so that $1-uv = \frac{1}{2}|v-u|^2$ and $u\bar\nabla v + v\bar\nabla u = -\frac{1}{2}\bar\nabla |v-u|^2$. Thus, we have
\begin{align*}
\psi& \geq - C{\epsilon_0}^{1/4}\int_{\mathcal{M}} |\bar\Delta (v-u)|^2  dv_{\bar g} + 2 \int_{\mathcal{M}} (\bar\Delta u + u|\bar\nabla u|^2) v(|\bar\nabla v|^2 - |\bar\nabla u|^2) dv_{\bar g} \\
& = - C{\epsilon_0}^{1/4} \int_{\mathcal{M}} |\bar\Delta (v-u)|^2 dv_{\bar g} + 2 \int_{\mathcal{M}}\bar\tau(u) \cdot (v-u) \langle\bar\nabla (v+u), \bar\nabla (v-u)\rangle \; dv_{\bar g} \\
&\geq - C{\epsilon_0}^{1/4} \int_{\mathcal{M}} |\bar\Delta (v-u)|^2  dv_{\bar g}\\
&\quad - 2 \left(\int_{\mathcal{M}}|\bar\tau(u)|^2|v-u|^2 dv_{\bar g}\right)^{\frac{1}{2}}
\left(\int_{\mathcal{M}}|\bar\nabla (v+u)|^4 dv_{\bar g}\right)^{\frac{1}{4}}\left(\int_{\mathcal{M}}|\bar\nabla (v-u)|^4 dv_{\bar g}\right)^{\frac{1}{4}}\\
&\geq - C{\epsilon_0}^{1/4} \int_{\mathcal{M}} |\bar\Delta (v-u)|^2 \; dv_{\bar g} \,,
\end{align*}
where we used \eqref{control-1}, \eqref{gradientbound2}, \eqref{gradientbound3}, Hardy inequality (Theorem \ref{secondhardy}), \eqref{smallenergy-2} and \eqref{smallenergy-3}.
\end{proof}

\begin{proof}[Proof of Theorem \ref{convexity}]
From \eqref{claim-2}, we have
\begin{equation}\label{laplace-com}
    \int_{\mathcal{M}}|\Delta v-\Delta u|^2dv_g\leq C\int_{\mathcal{M}}|\bar\Delta v-\bar\Delta u|^2 dv_{\bar g}\,.
\end{equation}
Combining \eqref{tau laplace} and \eqref{difference tension} we get
\begin{align}\label{taudiff}
 &\int_{\mathcal{M}}|\bar \tau(u)|^2dv_{\bar g}-\int_{\mathcal{M}}|\bar \tau(v)|^2dv_{\bar g}+\int_{\mathcal{M}}|\bar \tau(v)-\bar \tau(u)|^2dv_{\bar g}\notag \\
 \leq& \,C\epsilon_0^{1/4} \int_{\mathcal{M}}|\bar \tau(v)-\bar \tau(u)|^2dv_{\bar g}\,.
\end{align}
Now choosing $\epsilon_0$ sufficiently small we get (using \eqref{laplace-com} and \eqref{taudiff})
\begin{align*}
&\int_{\mathcal{M}}|\Delta v-\Delta u|^2dv_g \leq C\int_{\mathcal{M}}|\bar\Delta v-\bar\Delta u|^2 dv_{\bar g}\\
\leq & 4C\int_{\mathcal{M}}|\bar \tau(v)-\bar \tau(u)|^2dv_{\bar g} \leq 8C\left(\int_{\mathcal{M}}|\bar \tau(v)|^2dv_{\bar g}-\int_{\mathcal{M}}|\bar \tau(u)|^2dv_{\bar g}\right)\\
=&8C\sum_{i=1,\cdots, N}\left(\int_{\Omega_i}|\bar \tau(v)|^2dv_{\bar g}-\int_{\Omega_i}|\bar \tau(u)|^2dv_{\bar g}\right) = 8C\sum_{i=1,\cdots, N}\left(\mathcal{E}(v, \Omega_i) - \mathcal{E}(u, \Omega_i)\right)\\
=& 8C\left(\mathcal{E}(v, \mathcal{M}) - \mathcal{E}(u, \mathcal{M})\right)\,.
\end{align*}
This completes the proof of Theorem \ref{convexity}.
\end{proof}

\section{Conflict of interest and Data availability}
No conflict of interest exits in the submission of this manuscript,
and manuscript is approved by all authors for publication.
Data sharing not applicable to this article as no datasets were generated or analyzed in this study.

\appendix

\section{Conformal invariance of $\mathcal{E}(u,\mathcal{M})$}\label{appA}

In this appendix we give a quick check of the conformal invariance of $\mathcal{E}(u,\mathcal{M})$ in four dimensions which does not seem to be in the literature. Let $\mathcal{M}$ be a smooth Riemannian 4-manifold with or without boundary and $\mathcal{N}$ be another smooth Riemannian manifold. Given $u\in C^2(\mathcal{M},\mathcal{N})$, we denote by $u^*T\mathcal{N}$ the pull-back on $(\mathcal{M},g)$ of the tangent bundle of $(\mathcal{N},h)\hookrightarrow \mathbf{R}^K$. Then the tension field $\tau(u)\in u^*T\mathcal{N}$ is defined by 
\begin{equation}
\tau(u)=\sum_{\alpha}{\tilde\nabla}_{e_\alpha}du(e_\alpha),
\end{equation}
where $\{e_\alpha\}$ is an orthonormal frame of $T\mathcal{M}$ and $\tilde{\nabla}$ is the metric connection on $T^*\mathcal{M}\otimes u^*T\mathcal{N}$. Fix local coordinates $\{x^\alpha\}$ and $\{y^i\}$ of $\mathcal{M}$ and $\mathcal{N}$ respectively, we have
\begin{equation}
\tau^i(u)=g^{\alpha\beta}\frac{\partial^2u^i}{\partial x^\alpha\partial x^\beta}-g^{\alpha\beta}\Gamma^\gamma_{\alpha\beta}\frac{\partial u^i}{\partial x^\gamma}+g^{\alpha\beta}\bar{\Gamma}^i_{jk}\frac{\partial u^j}{\partial x^\alpha}\frac{\partial u^k}{\partial x^\beta},
\end{equation}
where $\Gamma$ and $\bar{\Gamma}$ are the Christoffel symbols on $\mathcal{M}$ and $\mathcal{N}$ respectively. Define $$A^i:=g^{\alpha\beta}\frac{\partial^2u^i}{\partial x^\alpha\partial x^\beta}+g^{\alpha\beta}\bar{\Gamma}^i_{jk}\frac{\partial u^j}{\partial x^\alpha}\frac{\partial u^k}{\partial x^\beta}.$$ 
Then 
\begin{equation}
\begin{split}
\tau^i(u)&=A^i-g^{\alpha\beta}\Gamma^\gamma_{\alpha\beta}\frac{\partial u^i}{\partial x^\gamma} =A^i-g^{\alpha\beta}\langle\nabla_{X_\alpha}X_\beta,du^i\rangle,
\end{split}
\end{equation}
where $\nabla= \nabla_g$ is the connection on $\mathcal{M}$ and $X_\alpha:=\frac{\partial}{\partial x^\alpha}$.
Now under the conformal change $\bar{g}=e^{2\phi}g$ (with $\phi\in C^{\infty}_0(\mathcal{M},\mathbf{R})$ if $\partial\mathcal{M}\neq\emptyset$), we have
\begin{equation}\label{A}
\bar{A}^i=e^{-2\phi}A^i
\end{equation}
and
\begin{equation}\label{christoffel}
\bar{g}^{\alpha\beta}\langle(\bar{\nabla}_{\bar{g}})_{X_\alpha}X_\beta,du^i\rangle=e^{-2\phi}g^{\alpha\beta}\langle\nabla_{X_\alpha}X_\beta,du^i\rangle-2e^{-2\phi}\langle\nabla\phi,du^i\rangle\,,
\end{equation}
where we have used 
\begin{equation}\label{conformalchange-1}
(\bar\nabla_{\bar{g}})_XY=\nabla_XY+X(\phi)Y+Y(\phi)X-g(X,Y)\nabla\phi.
\end{equation}
Therefore, combing \eqref{A} and \eqref{christoffel} together, we get
\begin{equation}
\bar\tau^i(u)=e^{-2\phi}\tau^i(u)+2e^{-2\phi}\langle\nabla\phi,du^i\rangle
\end{equation}
and
\begin{equation}\label{tau}
|\bar{\tau}(u)|^2-e^{-4\phi}|\tau(u)|^2=4e^{-4\phi}h_{ij}(u)\tau^j(u)\langle\nabla\phi,du^i\rangle+4e^{-4\phi}|\langle\nabla\phi,du\rangle|^2.
\end{equation}

Recall the following formulas for the conformal change of Ricci tensor and scalar curvature:
\begin{equation}
\overline{\rm Ric}={\rm Ric}-2({\rm Hess}(\phi)-\nabla\phi\otimes\nabla\phi)-(\Delta\phi+2|\nabla\phi|^2)g
\end{equation}
and 
\begin{equation}
\overline{\rm{Sc}}=e^{-2\phi}({\rm{Sc}}-6\Delta\phi-6|\nabla\phi|^2),
\end{equation}
where ${\rm Hess}(\phi)$ denotes the Hessian of $\phi$. Consequently, we obtain
\begin{align}\label{s}
\frac23\overline{\rm{Sc}}|du|_{\bar{g}}^2-\frac23e^{-4\phi}{\rm{Sc}}|du|^2_{g}&=\frac23e^{-4\phi}(-6\Delta\phi-6|\nabla\phi|^2)|du|^2_{g}\\
&=-4e^{-4\phi}(\Delta\phi+|\nabla\phi|^2)|du|^2_{g}\notag
\end{align}
and 
\begin{align}\label{ric}
&2\overline{\rm Ric}(du,du)-2e^{-4\phi}{\rm Ric}(du,du)\notag\\
=&2\bar{R}_{\alpha\beta}\bar{g}^{\alpha\gamma}\bar{g}^{\beta\sigma}\left\langle\frac{\partial u}{\partial x_\gamma},\frac{\partial u}{\partial x_\sigma}\right\rangle-2e^{-4\phi}{\rm Ric}(du,du)\notag\\
=&2e^{-4\phi}g^{\alpha\gamma}g^{\beta\sigma}\bigg((-\Delta\phi-2|\nabla\phi|^2)g_{\alpha\beta}-2(\nabla_\beta\phi_{\alpha}-\phi_\alpha\phi_\beta)\bigg)\langle u_\gamma,u_\sigma\rangle\notag\\
=&-4e^{-4\phi}|\nabla\phi|^2|du|_g^2+4e^{-4\phi}|\langle\nabla\phi,du\rangle|^2-2e^{-4\phi}\Delta\phi|du|_g^2\\\
&-4e^{-4\phi}\sum_{\alpha, \beta}\nabla^2_{\alpha\beta}\phi\left\langle\nabla_\alpha u,\nabla_\beta u\right\rangle\,,\notag
\end{align}
where Einstein summation convention was used and we used different subscripts to distinguish different metrics on the bundle $T^*\mathcal{M}\otimes u^*T\mathcal{N}$. Combining these we have
\begin{align}
&\int_\mathcal{M}\left[\eqref{tau}+\eqref{s}-\eqref{ric}\right]dv_{\bar{g}}\notag\\
=&\sum_{\alpha, \beta}\int_\mathcal{M} 2e^{-4\phi}\left(2h_{ij}(u)\tau^j(u)\langle\nabla\phi,du^i\rangle-\Delta\phi|du|_g^2+2\nabla^2_{\alpha\beta}\phi\langle\nabla_\alpha u,\nabla_\beta u\rangle\right) dv_{\bar{g}}\notag\\
=&{\sum_{\alpha, \beta}\int_\mathcal{M} 4\nabla_\alpha \phi \langle \Delta u,\nabla_\alpha{u}\rangle-2\Delta\phi|du|_g^2+4\nabla^2_{\alpha\beta}\phi\langle\nabla_\alpha u,\nabla_\beta u\rangle dv_{g}}\notag\\
=&\sum_{\alpha, \beta}\int_\mathcal{M} -2\Delta\phi|du|^2-4\nabla_{\beta}\phi\langle\nabla_\alpha u,\nabla_{\alpha\beta}^2u\rangle dv_{g}\notag\\
=&\sum_{\alpha, \beta}\int_\mathcal{M} 2\nabla_\beta{\phi}\nabla_\beta\langle\nabla_\alpha u,\nabla_\alpha u\rangle-4\nabla_{\beta}\phi\langle\nabla_\alpha u,\nabla^2_{\alpha\beta}u\rangle dv_{g}=0.\notag
\end{align}
Hence, the energy functional $\mathcal{E}(u,\mathcal{M})$ is conformally invariant in four dimensions. 
Now using the notations in \eqref{conformflat}, on each $\Omega_i, i=1,\cdots, N$ we have
$$g=e^{2\phi_i}\bar{g}\,,$$
where $\bar{g}$ is the standard Euclidean metric. By the above conformal change of $\mathcal{E}(\cdot, \mathcal{M})$, for any $w\in C^2(\mathcal{M},\mathcal{N})$ we get
\begin{align*}
&\mathcal{E}(w, (\mathcal{M},g)) \\
= &\sum_{i}\mathcal{E}(w, (\Omega_i,\bar{g}))+ \sum_{i,\alpha, \beta}\int_{\Omega_i} 4\bar\nabla_\alpha \phi_i \langle \bar\Delta w,\bar\nabla_\alpha w\rangle-2\bar\Delta\phi_i|\bar\nabla w|^2+4\bar\nabla^2_{\alpha\beta}\phi_i\langle\bar\nabla_\alpha w,\bar\nabla_\beta w\rangle dv_{\bar{g}}.
\end{align*}
Thus, for any variation $w(t)$ of $w$ such that $\left.\frac{d}{dt}\right|_{t=0}w(t)=P\phi\in T_w\mathcal{N}$ where $\phi \in C^\infty_0(\mathcal{M}, \mathbf{R}^K)$, the first variation formula for $\mathcal{E}(w, (\mathcal{M},g))$ is (cf. Appendix \ref{appB})
\begin{align}
&\left.\frac{d}{dt}\right|_{t=0}{\mathcal{E}}(w(t),(\mathcal{M},g))\label{firstvariation-1}\\
=&2\sum_i\int_{\Omega_i} e^{-4\phi_i}\big\langle P(\bar\Delta^2u)-P(A(\bar\nabla u,\bar\nabla u){D}_u A(\bar\nabla u,\bar\nabla u))\notag\\
&+2{\rm div}(A(\bar\nabla u,\bar\nabla u)A(\bar\nabla u,P))-2A(\bar\nabla u,\bar\nabla u)A(\bar\nabla u,\bar\nabla P), \phi\big\rangle_g dv_g\notag\\
&+\sum_{i,\alpha, \beta}\int_{\Omega_i} \big[4\bar\nabla_\alpha \phi_i \langle \bar\Delta (P\phi),\bar\nabla_\alpha w\rangle+ 4\bar\nabla_\alpha \phi_i \langle \bar\Delta w,\bar\nabla_\alpha (P\phi)\rangle\label{fir1}\\
&-4\bar\Delta\phi_i\langle\bar\nabla w, \nabla (P\phi)\rangle+8\bar\nabla^2_{\alpha\beta}\phi_i\langle\bar\nabla_\alpha (P\phi),\bar\nabla_\beta w\rangle \big]dv_{\bar{g}}\label{fir2}\,.
\end{align}
In particular, for any fixed $j=1,\cdots, N$ and any $\phi \in C^\infty_0(\Omega_j, \mathbf{R}^K)$ we have
\begin{align}\label{fir3}
&\left.\frac{d}{dt}\right|_{t=0}{\mathcal{E}}(w(t),(\mathcal{M},g))\notag\\
=&2\int_{\Omega_j} e^{-4\phi_j}\big\langle P(\bar\Delta^2u)-P(A(\bar\nabla u,\bar\nabla u){D}_u A(\bar\nabla u,\bar\nabla u))\\
&+2{\rm div}(A(\bar\nabla u,\bar\nabla u)A(\bar\nabla u,P))-2A(\bar\nabla u,\bar\nabla u)A(\bar\nabla u,\bar\nabla P), \phi\big\rangle_g dv_g\,,\notag
\end{align}
noting that since $\phi \in C^\infty_0(\Omega_j, \mathbf{R}^K)$, the terms in \eqref{fir1} and \eqref{fir2} cancel out to zero using integration by parts over $\Omega_j$. Therefore we have
\begin{prop}\label{local property}
If $u$ is a conformal-harmonic map on a locally conformally
flat $4$-manifold $(\mathcal{M},g)$, then $u$ is locally an intrinsic biharmonic map on $\mathcal{M}$ (with respect to the flat connection) up to a conformal change of $g$.
\end{prop}


\section{Conformal-harmonic map equation}\label{appB}

For the sake of completeness, in this appendix we include a derivation of the conformal-harmonic map equation \eqref{ELP}, i.e., the Euler-Lagrangian equation for $\mathcal{E}(u,\mathcal{M})$. Consider a variation $u(t)$ of $u$ such that $\left.\frac{d}{dt}\right|_{t=0}u(t)=P\phi\in T_u\mathcal{N}$ where $\phi \in C^\infty_0(\mathcal{M}, \mathbf{R}^K)$. Then 
\begin{align}\label{1stvariation}
&\left.\frac{d}{d t}\right|_{t=0}\mathcal{E}(u(t),\mathcal{M})\notag\\
=&\int_\mathcal{M}\frac{d}{d t}\bigg|_{t=0}|\tau(u(t))|^2+\frac23 {\rm Sc}^{\mathcal{M}}\frac{d}{ d t}\bigg|_{t=0}|d u(t)|^2-2\frac{d}{d t}\bigg|_{t=0}{\rm Ric}^\mathcal{M}(d u(t), d u(t))\notag\\
:=&\,{\textbf{I}}+{\textbf{II}}+{\textbf{III}}.
\end{align}
We compute the three terms in \eqref{1stvariation} as follow. Note that the first term {\textbf{I}} is the first variation of the intrinsic bienergy $\int_{\mathcal{M}}|\tau(u)|^2$ and therefore yields exactly the intrinsic biharmonic map equation.
\begin{align}\label{fir4-biharmonic}
{\textbf{I}}=&\frac{d}{d t}\bigg|_{t=0}\int_\mathcal{M} \left|\Delta u(t)|^2-|A(u(t))(\nabla u(t),\nabla u(t)) \right|^2\notag\\
=&2\int_\mathcal{M}\langle\Delta u,\Delta(P\phi)\rangle-A(\nabla u,\nabla u)(D_u A(\nabla u,\nabla u)P\phi\notag+2A(\nabla u,\nabla(P\phi)))\\
=&2\int_\mathcal{M}\langle\Delta^2u,P\phi\rangle-\langle A(\nabla u,\nabla u)D_u A(\nabla u,\nabla u),P\phi\rangle \notag\\
&+2\langle\nabla(A(\nabla u,\nabla u)A(\nabla u,\cdot)),P\phi\rangle\notag\\
=&2\int_\mathcal{M}\langle P(\Delta^2u)-P(A(\nabla u,\nabla u)D_u A(\nabla u,\nabla u))\\
&+2{\rm div}(A(\nabla u,\nabla u)A(\nabla u,P))-2A(\nabla u,\nabla u)A(\nabla u,\nabla P),\phi\rangle.\notag
\end{align}
Now let's look at terms {\textbf{II}} and {\textbf{III}} which give the additional lower order terms to the intrinsic biharmonic map equation that make the conformal-harmonic map equation \eqref{ELP} conformally invariant.

\begin{align}
{\textbf{II}}&=\frac23\int_\mathcal{M} {\rm Sc}^\mathcal{M}\frac{d}{d t}\bigg|_{t=0}|d u(t)|^2\notag =\frac43\int_\mathcal{M}\langle {\rm Sc}^\mathcal{M} \nabla u, \nabla{(P\phi)}\rangle\\ &=\frac43\int_\mathcal{M}\langle {\rm Sc}^\mathcal{M}\nabla_{e_\alpha}u,\nabla_{e_\alpha}(P\phi)\rangle =-\frac43\int_\mathcal{M}\langle \nabla_{e_\alpha}({\rm Sc}^\mathcal{M} \nabla_{e_\alpha}u),P\phi\rangle\notag\\
&=-\frac43\int_\mathcal{M}\langle {\rm Sc}^\mathcal{M}\Delta u+\nabla {\rm Sc}^\mathcal{M}\nabla u,P\phi\rangle \notag\\
&=-\frac43\int_\mathcal{M}\langle {\rm Sc}^\mathcal{M}(\Delta u+A(u)(\nabla{u},\nabla{u}))+\nabla {\rm Sc}^\mathcal{M}\nabla u,\phi\rangle.
\end{align}

\begin{align}
{\textbf{III}}
&=-2\int_\mathcal{M}\frac{d}{dt}\bigg|_{t=0}{\rm Ric}^\mathcal{M}(d u(t),d u(t))\notag\\
&=-4\int_\mathcal{M}{\rm Ric}^\mathcal{M}(d u,d(P\phi))\notag\\
&=-4\int_\mathcal{M}\langle{{\rm Ric}^\mathcal{M}(du,\cdot),d(P\phi)}\rangle\notag\\
&=4\int_\mathcal{M}\langle\nabla_{e_\beta}({\rm Ric}^\mathcal{M})^{\alpha\beta}\nabla_{e_\alpha}u+({\rm Ric}^\mathcal{M})^{\alpha\beta}\nabla^2_{\alpha\beta}{u},P\phi\rangle\notag\\
&=4\int_\mathcal{M}\langle\nabla_{\beta}({\rm Ric}^\mathcal{M})^{\alpha\beta}\nabla_{\alpha}u+({\rm Ric}^\mathcal{M})^{\alpha\beta}P(\nabla^2_{\alpha\beta}{u}),\phi\rangle,
\end{align}
where we used the subscripts $\alpha,\beta$ to denote $e_\alpha, e_\beta$ for short.


Combining these together, we get
\begin{align}
\frac{d}{d t}\bigg|_{t=0}{\mathcal{E}}(u(t),\mathcal{M})
=&2\int_\mathcal{M}\big\langle P(\Delta^2u)-P(A(\nabla u,\nabla u){D}_u A(\nabla u,\nabla u))\notag\\
&+2{\rm div}(A(\nabla u,\nabla u)A(\nabla u,P))-2A(\nabla u,\nabla u)A(\nabla u,\nabla P)\notag\\
&-\frac23{\rm Sc}^\mathcal{M}(\Delta u+A(u)(\nabla{u},\nabla{u}))-\frac23\nabla {\rm Sc}^\mathcal{M}\nabla u\\
&+2\nabla_{\beta}({\rm Ric}^\mathcal{M})^{\alpha\beta}\nabla_{\alpha}u+2({\rm Ric}^\mathcal{M})^{\alpha\beta}P(\nabla^2_{\alpha\beta}{u}),\phi\big\rangle.\notag
\end{align}

Thus, the critical point $u$ of $\mathcal{E}$ satisfies the fourth order PDE:
\begin{align}\label{horizontal}
P(\Delta^2u)
=&P(A(\nabla u,\nabla u){D}_u A(\nabla u,\nabla u))-2{\rm div}(A(\nabla u,\nabla u)A(\nabla u,P))\notag\\
&+2A(\nabla u,\nabla u)A(\nabla u,\nabla P)+\frac23{\rm Sc}^\mathcal{M}(\Delta u+A(u)(\nabla{u},\nabla{u}))\\
&+\frac23\nabla {\rm Sc}^\mathcal{M}\nabla u-2\nabla_{\beta}({\rm Ric}^\mathcal{M})^{\alpha\beta}\nabla_{\alpha}u-2({\rm Ric}^\mathcal{M})^{\alpha\beta}P(\nabla^2_{\alpha\beta}{u}).\notag
\end{align}

Note that
\begin{align}
P^\bot(\Delta^2u) &= \text{div}(P^\bot \nabla \Delta u) - \nabla P^\bot \nabla \Delta u\notag\\
&= - \Delta A(\nabla u, \nabla u)-\Delta P^\bot\Delta u - 2\nabla P^\bot\nabla \Delta u\notag\\
&= - \Delta A(\nabla u, \nabla u)+\Delta P\Delta u + 2\nabla P\nabla \Delta u\\
&=-\Delta(A(\nabla u,\nabla u))-\Delta P\Delta u+2\text{div}(\nabla P\Delta u)\,,\notag
\end{align}
and 
\begin{align*}
\Delta^2u &= P(\Delta^2u) + P^\bot(\Delta^2u)\\
&=P(\Delta^2u)-\Delta(A(\nabla u,\nabla u))-\Delta P\Delta u+2\text{div}(\nabla P\Delta u)\,.
\end{align*}

Therefore, \eqref{horizontal} can be rewritten as
\begin{align}
\Delta^2u =&-\Delta(A(\nabla u,\nabla u))-\Delta P\Delta u +2{\rm div}(\nabla P\Delta u)\notag\\
&+P(A(\nabla u,\nabla u){D}_u A(\nabla u,\nabla u))-2{\rm div}(A(\nabla u,\nabla u)A(\nabla u,P))\notag\\
&+2A(\nabla u,\nabla u)A(\nabla u,\nabla P)+\frac23{\rm Sc}^\mathcal{M}(\Delta u+A(u)(\nabla{u},\nabla{u}))\\
&+\frac23\nabla {\rm Sc}^{\mathcal{M}}\nabla u-2\nabla\text{Ric}^\mathcal{M}(\nabla u, \cdot)-2{\rm Ric}^\mathcal{M}(P(\nabla^2u), \cdot)\notag
\end{align}
or equivalently (c.f. \cite[equation (1.2)]{LL21}, \cite[equation (9)]{LaR})
\begin{align}
\Delta^2u  = &-\Delta(\nabla P^\bot\nabla u)-{\rm div}(\nabla P^\bot\Delta u)+2\nabla P^\bot\nabla(\nabla P^\bot\nabla u)\notag\\
&+2\nabla P^\bot\nabla P^\bot\Delta u-(\nabla PP^\bot-P^\bot\nabla P)\nabla\Delta u\notag\\
&+P\langle\nabla P^\bot\nabla u,{D}_u({D}_u P^\bot)\nabla u\nabla u\rangle-2{\rm div}\langle\nabla P^\bot\nabla u,\nabla P^\bot P\rangle\\
&+2\langle\nabla P^\bot\nabla u,\nabla P^\bot\nabla P\rangle+\frac23 {\rm Sc}^\mathcal{M}(\Delta u+\nabla P^\perp\nabla u) \notag\\
&+\frac23\nabla {\rm Sc}^{\mathcal{M}}\nabla u-2\nabla\text{Ric}^\mathcal{M}(\nabla u, \cdot)-2{\rm Ric}^\mathcal{M}(P(\nabla^2u), \cdot)\,,\notag
\end{align}
where
$$
\nabla\text{Ric}^\mathcal{M}(\nabla u, \cdot):=\nabla_{\beta}({\rm Ric}^\mathcal{M})^{\alpha\beta}\nabla_{\alpha}u
$$
and
$$
{\rm Ric}^\mathcal{M}(P(\nabla^2u), \cdot):=({\rm Ric}^\mathcal{M})^{\alpha\beta}P(\nabla^2_{\alpha\beta}{u})\,.
$$

\bibliographystyle{amsplain}
\bibliography{reference}

\providecommand{\bysame}{\leavevmode\hbox to3em{\hrulefill}\thinspace}
\providecommand{\MR}{\relax\ifhmode\unskip\space\fi MR }
\providecommand{\MRhref}[2]{%
  \href{http://www.ams.org/mathscinet-getitem?mr=#1}{#2}
}
\providecommand{\href}[2]{#2}
\begin{thebibliography}{10}

\bibitem{AY17}
Wanjun Ai and Hao Yin, \emph{Neck analysis of extrinsic polyharmonic maps},
  Ann. Global Anal. Geom. \textbf{52} (2017), no.~2, 129--156. \MR{3690012}

\bibitem{Berard08}
Vincent B\'{e}rard, \emph{Un analogue conforme des applications harmoniques},
  C. R. Math. Acad. Sci. Paris \textbf{346} (2008), no.~17-18, 985--988.
  \MR{2449641}

\bibitem{Biquard}
Olivier Biquard and Farid Madani, \emph{A construction of conformal-harmonic
  maps}, Comptes Rendus Mathematique \textbf{350} (2012), no.~21, 967--970.

\bibitem{Car97}
G.~Carron, \emph{In\'{e}galit\'{e}s de {H}ardy sur les vari\'{e}t\'{e}s
  riemanniennes non-compactes}, J. Math. Pures Appl. (9) \textbf{76} (1997),
  no.~10, 883--891. \MR{1489943}

\bibitem{CWY99}
S.~A. Chang, L.~Wang, and P.~C. Yang, \emph{{A regularity theory of biharmonic
  maps}}, Comm. Pure Appl. Math. \textbf{52} (1999), no.~9, 1113--1137.
  \MR{1692148}

\bibitem{CY99}
Sun-Yung~A. Chang and Paul~C. Yang, \emph{On a fourth order curvature
  invariant}, Spectral problems in geometry and arithmetic ({I}owa {C}ity,
  {IA}, 1997), Contemp. Math., vol. 237, Amer. Math. Soc., Providence, RI,
  1999, pp.~9--28. \MR{1710786}

\bibitem{MR2365237}
Bennett Chow, Sun-Chin Chu, David Glickenstein, Christine Guenther, James
  Isenberg, Tom Ivey, Dan Knopf, Peng Lu, Feng Luo, and Lei Ni, \emph{The
  {R}icci flow: techniques and applications. {P}art {II}}, Mathematical Surveys
  and Monographs, vol. 144, American Mathematical Society, Providence, RI,
  2008, Analytic aspects. \MR{2365237}

\bibitem{CM08}
T.~H. Colding and W.~P. Minicozzi, \emph{{Width and finite extinction time of
  {R}icci flow}}, Geom. Topol. \textbf{12} (2008), no.~5, 2537--2586.
  \MR{2460871}

\bibitem{DD14}
Lorenzo D'Ambrosio and Serena Dipierro, \emph{Hardy inequalities on
  {R}iemannian manifolds and applications}, Ann. Inst. H. Poincar\'{e} C Anal.
  Non Lin\'{e}aire \textbf{31} (2014), no.~3, 449--475. \MR{3208450}

\bibitem{LFG21}
Fr\'{e}d\'{e}ric~Louis de~Longueville and Andreas Gastel, \emph{Conservation
  laws for even order systems of polyharmonic map type}, Calc. Var. Partial
  Differential Equations \textbf{60} (2021), no.~4, Paper No. 138, 18.
  \MR{4279397}

\bibitem{DK98}
Frank Duzaar and Ernst Kuwert, \emph{Minimization of conformally invariant
  energies in homotopy classes}, Calc. Var. Partial Differential Equations
  \textbf{6} (1998), no.~4, 285--313. \MR{1624288}

\bibitem{ER99}
D.~E. Edmunds and J.~{R{\'a}kosn{\'\i} k}, \emph{{On a higher-order {H}ardy
  inequality}}, Math. Bohem. \textbf{124} (1999), no.~2-3, 113--121.
  \MR{1780685}

\bibitem{ES64}
James Eells, Jr. and J.~H. Sampson, \emph{Harmonic mappings of {R}iemannian
  manifolds}, Amer. J. Math. \textbf{86} (1964), 109--160. \MR{164306}

\bibitem{Gas06}
A.~Gastel, \emph{{The extrinsic polyharmonic map heat flow in the critical
  dimension}}, Adv. Geom. \textbf{6} (2006), no.~4, 501--521. \MR{2267035}

\bibitem{GN13}
Andreas Gastel and Andreas~J. Nerf, \emph{Minimizing sequences for conformally
  invariant integrals of higher order}, Calc. Var. Partial Differential
  Equations \textbf{47} (2013), no.~3-4, 499--521. \MR{3070553}

\bibitem{GS}
Andreas Gastel and Christoph Scheven, \emph{Regularity of polyharmonic maps in
  the critical dimension}, Comm. Anal. Geom. \textbf{17} (2009), no.~2,
  185--226. \MR{2520907}

\bibitem{GJMS92}
C.~Robin Graham, Ralph Jenne, Lionel~J. Mason, and George A.~J. Sparling,
  \emph{Conformally invariant powers of the {L}aplacian. {I}. {E}xistence}, J.
  London Math. Soc. (2) \textbf{46} (1992), no.~3, 557--565. \MR{1190438}

\bibitem{HJL19}
Weiyong He, Ruiqi Jiang, and Longzhi Lin, \emph{Existence of polyharmonic maps
  in critical dimensions}, Preprint (2019).

\bibitem{He02}
F.~H{\'e}lein, \emph{{Harmonic maps, conservation laws and moving frames}},
  second ed., {Cambridge Tracts in Mathematics}, vol. 150, Cambridge University
  Press, Cambridge, 2002, Translated from the 1996 French original, With a
  foreword by James Eells. \MR{1913803}

\bibitem{HHW14}
J.~Hineman, T.~Huang, and C.~Wang, \emph{{Regularity and uniqueness of a class
  of biharmonic map heat flows}}, Calc. Var. Partial Differential Equations
  \textbf{50} (2014), no.~3-4, 491--524. \MR{3216822}

\bibitem{HLLW16}
T.~Huang, L.~Liu, Y.~Luo, and C.~Wang, \emph{{Heat flow of extrinsic biharmonic
  maps from a four dimensional manifold with boundary}}, J. Elliptic Parabol.
  Equ. \textbf{2} (2016), no.~1-2, 1--26. \MR{3645932}

\bibitem{Jiang86}
G.~Jiang, \emph{{{$2$}-harmonic isometric immersions between {R}iemannian
  manifolds}}, Chinese Ann. Math. Ser. A \textbf{7} (1986), no.~2, 130--144, An
  English summary appears in Chinese Ann. Math. Ser. B {{\bf{7}}} (1986), no.
  2, 255.

\bibitem{KM97}
Juha Kinnunen and Olli Martio, \emph{Hardy's inequalities for {S}obolev
  functions}, Math. Res. Lett. \textbf{4} (1997), no.~4, 489--500. \MR{1470421}

\bibitem{KO09}
Ismail Kombe and Murad \"{O}zaydin, \emph{Improved {H}ardy and {R}ellich
  inequalities on {R}iemannian manifolds}, Trans. Amer. Math. Soc. \textbf{361}
  (2009), no.~12, 6191--6203. \MR{2538592}

\bibitem{Ku08}
Y.~Ku, \emph{{Interior and boundary regularity of intrinsic biharmonic maps to
  spheres}}, Pacific J. Math. \textbf{234} (2008), no.~1, 43--67. \MR{2375314}

\bibitem{Lamm04}
T.~Lamm, \emph{{Heat flow for extrinsic biharmonic maps with small initial
  energy}}, Ann. Global Anal. Geom. \textbf{26} (2004), no.~4, 369--384.
  \MR{2103406}

\bibitem{Lamm05}
\bysame, \emph{{Biharmonic map heat flow into manifolds of nonpositive
  curvature}}, Calc. Var. Partial Differential Equations \textbf{22} (2005),
  no.~4, 421--445. \MR{2124627}

\bibitem{LL13}
T.~Lamm and L.~Lin, \emph{{Estimates for the energy density of critical points
  of a class of conformally invariant variational problems}}, Adv. Calc. Var.
  \textbf{6} (2013), no.~4, 391--413. \MR{3199733}

\bibitem{LR}
T.~Lamm and T.~Rivi{\`e}re, \emph{{Conservation laws for fourth order systems
  in four dimensions}}, Comm. Partial Differential Equations \textbf{33}
  (2008), no.~1-3, 245--262. \MR{2398228}

\bibitem{LW09}
Tobias Lamm and Changyou Wang, \emph{Boundary regularity for polyharmonic maps
  in the critical dimension}, Adv. Calc. Var. \textbf{2} (2009), no.~1, 1--16.
  \MR{2494504}

\bibitem{LaR}
P.~Laurain and T.~Rivi{\`e}re, \emph{{Energy quantization for biharmonic
  maps}}, Adv. Calc. Var. \textbf{6} (2013), no.~2, 191--216. \MR{3043576}

\bibitem{LL21}
Paul Laurain and Longzhi Lin, \emph{Energy convexity of intrinsic bi-harmonic
  maps and applications {I}: {S}pherical target}, J. Reine Angew. Math.
  \textbf{772} (2021), 53--81. \MR{4227593}

\bibitem{LP19}
Paul Laurain and Romain Petrides, \emph{Existence of min-max free boundary
  disks realizing the width of a manifold}, Adv. Math. \textbf{352} (2019),
  326--371. \MR{3961741}

\bibitem{LSZ20}
Longzhi Lin, Ao~Sun, and Xin Zhou, \emph{Min-max minimal disks with free
  boundary in {R}iemannian manifolds}, Geom. Topol. \textbf{24} (2020), no.~1,
  471--532. \MR{4080488}

\bibitem{LY16}
Lei Liu and Hao Yin, \emph{Neck analysis for biharmonic maps}, Math. Z.
  \textbf{283} (2016), no.~3-4, 807--834. \MR{3519983}

\bibitem{Mit00}
\`E. Mitidieri, \emph{A simple approach to {H}ardy inequalities}, Mat. Zametki
  \textbf{67} (2000), no.~4, 563--572. \MR{1769903}

\bibitem{Mos05}
R.~Moser, \emph{{The blowup behavior of the biharmonic map heat flow in four
  dimensions}}, IMRP Int. Math. Res. Pap. (2005), no.~7, 351--402.

\bibitem{Mos08}
Roger Moser, \emph{A variational problem pertaining to biharmonic maps}, Comm.
  Partial Differential Equations \textbf{33} (2008), no.~7-9, 1654--1689.
  \MR{2450176}

\bibitem{Pa83}
Stephen~M. Paneitz, \emph{A quartic conformally covariant differential operator
  for arbitrary pseudo-{R}iemannian manifolds (summary)}, SIGMA Symmetry
  Integrability Geom. Methods Appl. \textbf{4} (2008), Paper 036, 3.
  \MR{2393291}

\bibitem{Sche08}
Christoph Scheven, \emph{Dimension reduction for the singular set of biharmonic
  maps}, Adv. Calc. Var. \textbf{1} (2008), no.~1, 53--91. \MR{2402212}

\bibitem{Str85}
Michael Struwe, \emph{On the evolution of harmonic mappings of {R}iemannian
  surfaces}, Comment. Math. Helv. \textbf{60} (1985), no.~4, 558--581.
  \MR{826871}

\bibitem{Stru08}
\bysame, \emph{Partial regularity for biharmonic maps, revisited}, Calc. Var.
  Partial Differential Equations \textbf{33} (2008), no.~2, 249--262.
  \MR{2413109}

\bibitem{Strz03}
Pawe\l Strzelecki, \emph{On biharmonic maps and their generalizations}, Calc.
  Var. Partial Differential Equations \textbf{18} (2003), no.~4, 401--432.
  \MR{2020368}

\bibitem{Wang04-2}
C.~Wang, \emph{{Stationary biharmonic maps from {$\Bbb R^m$} into a
  {R}iemannian manifold}}, Comm. Pure Appl. Math. \textbf{57} (2004), no.~4,
  419--444. \MR{2026177}

\bibitem{Wang04}
Changyou Wang, \emph{Biharmonic maps from {$\bold R^4$} into a {R}iemannian
  manifold}, Math. Z. \textbf{247} (2004), no.~1, 65--87. \MR{2054520}

\bibitem{Wang07}
\bysame, \emph{Heat flow of biharmonic maps in dimensions four and its
  application}, Pure Appl. Math. Q. \textbf{3} (2007), no.~2, Special Issue: In
  honor of Leon Simon. Part 1, 595--613. \MR{2340056}

\end{thebibliography}
\end{document}